\documentclass[review]{elsarticle}

\usepackage{lineno,hyperref}

\journal{Mathematics and Computers in Simulation}
\usepackage{amssymb}
\usepackage{amsmath}
\usepackage{amsthm}
\usepackage{epstopdf}
\usepackage{mathtools} 
\usepackage{subfig}
\usepackage{float}
\usepackage{mathrsfs}
\usepackage{upgreek} 
\newsubfloat{figure}
\numberwithin{figure}{section}
\newcommand{\be}{\begin{equation}}
\newcommand{\ee}{\end{equation}}
\newcommand{\bse}{\begin{subequations}}
\newcommand{\ese}{\end{subequations}}
\newcommand{\no}{\nonumber}
\newcommand{\ep}{\epsilon}

\newcommand{\fa}{\forall}

\def\ni{\noindent}
\newtheorem{theorem}{Theorem}[section]
\newtheorem{lemma}[theorem]{Lemma}

\newtheorem{example}[theorem]{Example}

\begin{document}

\begin{frontmatter}

\title{Complete Flux Scheme for Elliptic Singularly Perturbed Differential-Difference Equations}


\author[]{Sunil Kumar$^a$\corref{mycorrespondingauthor}}
\cortext[mycorrespondingauthor]{Corresponding author}
\ead{sunilmath2015@gmail.com}

\author[]{B.V. Rathish Kumar$^b$}
\ead{bvrk@iitk.ac.in}

\author[]{J.H.M. Ten Thije Boonkkamp$^c$}
\ead{j.h.m.tenthijeboonkkamp@tue.nl}

\address{$^a$Department of Mathematics \\ Deshbandhu College, Delhi University\\ New Delhi, India}
\address{$^b$Department of Mathematics and Statistics\\ Indian Institute of Technology Kanpur\\ Kanpur, India}
\address{$^c$Department of Mathematics and Computer Science\\ Eindhoven University of Technology\\ Eindhoven, Netherlands}

\begin{abstract}
In this study, we propose a new scheme named as complete flux scheme (CFS) based on the finite volume method for solving singularly perturbed differential-difference equations (SPDDEs) of elliptic type. An alternate integral representation for the flux is obtained which plays an important role in the derivation of CF scheme. We have established the stability, consistency and quadrature convergence of the proposed scheme. The scheme is successfully implemented on test problems.
\end{abstract}


\end{frontmatter}


\section{Introduction}
\label{intro}

Many conservation laws which occur frequently in fluid mechanics, combustion theory, plasma physics and semiconductor physics etc., are of advection-diffusion-reaction type (in particular singularly perturbed type) and describe the interplay between different processes such as advection or drift, diffusion or conduction and reaction or recombination, see \cite{pove05,jero96,mari90}. \emph{Singularly perturbed problems} (SPPs) are special type of differential equations which are different in nature of its solutions and complicated to solve by numerical methods \cite{rost96}. SPPs are those problems that show very rapid change in its solutions, i.e., such problems where some variables vary much faster than the other variables. In an SPP, a very small parameter $\ep$ called singular perturbation parameter, is multiplied to the highest order derivative term and as this parameter goes smaller and smaller, boundary layer occurs. Then the solution shows a very abrupt change in a very small portion of the domain. In such a small portion, it becomes challenging for the numerical methods to capture the solution accurately particularly in the layer region when the singular perturbation parameter tends to zero.

A \emph{differential-difference equation} (DDE) is a differential equation for which the evolution not only depends on the current state of the system but also on the past history \cite{liwd05}. In a DDE, a small positive parameter $\mu$ (say) is subtracted or added to one or more arguments of the unknown function $u$ or its derivatives. This parameter $\mu$ is said to be a negative shift, if it is subtracted while this parameter $\mu$ is said to be a positive shift if it is added. A differential equation is said be a \emph{singularly perturbed differential-difference equation} (SPDDE) if it has the characteristics of both SPPs and DDEs \cite{pna89}.
SPDDEs are not easy to solve by the usual numerical methods because usual methods are often not capable to capture boundary layers accurately. Thus, for solving these equations something more is needed like comparatively finer mesh or Shishkin mesh etc. There are various numerical methods available for space discretization like finite element, finite volume, finite difference and spectral methods etc.

Here, we present a new scheme in finite volume framework, namely \emph{complete flux scheme} (CFS) for solving the elliptic singularly perturbed differential-difference equations (SPDDEs). Finite volume methods (FVMs) are based on the integral formulation, i.e., an integration is performed on the conservation law over a disjunct set of control volumes that cover the domain. Complete flux schemes are based on the integral representations of the fluxes that play an important role to obtain numerical flux approximations. In complete flux scheme, the fluxes are computed by using the source term. For solving advection-diffusion-reaction type problems, the complete flux scheme was given by Ten Thije Boonkkamp and Anthonissen \cite{bnkmp11}. The complete flux scheme is an extension of the exponential schemes of Thiart \cite{gdth17,gdth18}. An integral representation for the flux from the solution of a local BVP has been obtained for the entire equation. The two components of the flux, namely homogeneous and inhomogeneous correspond to the homogeneous and the particular solution of the BVP, respectively. The idea of representing the solution in two adjacent intervals in terms of an approximate Green's function \cite{kwm96}, is used to obtain the inhomogeneous flux. Suitable quadrature rules when applied to the integral representation of flux lead to complete flux schemes.

The complete flux scheme is second order accurate in space, in particular, the flux approximations remain second order accurate for highly dominant advection and do not produce spurious oscillations for dominant advection. Also, the flux approximations only depend on neighbouring values resulting in a scheme limited to local neighbourhood and thereby avoids need for higher resolution. Also, the source terms are included in the computation of the fluxes to ensure conservation law at a discrete level. From the current literature on CFS, one can easily find that apart from the recent contribution \cite{bnkmp11,likmp13,hof98} of Ten Thije Boonkkamp and a few of his co-authors, there are hardly any reportings related to the work based on CFS. Hence, here a very first attempt has been made to explore the effectiveness of CFS in capturing the boundary layers associated with the elliptic SPDDEs for the first time.

The paper is organized under seven sections. Introductory remarks on complete flux scheme for elliptic SPDDEs are made in Section \ref{intro}. A detailed description of complete flux scheme is presented in Section \ref{sec:5.3}. In Sections \ref{sec:5.4}-\ref{sec:5.6}, stability, consistency and convergence is established, respectively. Further in Section \ref{sec:5.7}, we have successfully implemented complete flux scheme on some example problems. In the last Section \ref{sec:5.8}, conclusions have been given.

\section{The Continuous Problem}
\label{sec:5.3}
Consider the following BVP for the elliptic SPDDE
\bse \label{eq:5.3.1}
\begin{align}
- \ep\phi^{\prime\prime}(x) + b\phi^{\prime}(x-\mu)= s, ~ \fa x \in\Omega, \hspace{2cm}& \label{eq:5.3.1a} \\
u(x) = q(x) \geq 0,~-\mu \leq x < 0,~~ \phi(1) = q(1) \geq 0,\hspace{0.8cm}& \label{eq:5.3.1b}
\end{align}
\ese
where $0 < \ep \ll 1$ is the singular perturbation parameter, $\mu $ is a small shift argument of $O(\ep)$ such that $\mu \geq 0$, $b$ is a constant on domain $\Omega = (0,1)$, $s$ is a source term and the prime $(^{\prime})$ is the differentiation with respect to $x$. The source term $s$ can be a constant or function of $x$ and $\phi$. Here, for the sake of discretization, we consider $s$ to be a function of $x$. 

For small $\mu$, the following BVP is a good approximation to (\ref{eq:5.3.1})
\bse \label{eq:5.3.111}
\begin{align}
\mathscr{L}{\phi} \coloneqq ~&- (\ep + \mu b) \phi^{\prime\prime} + b \phi^{\prime} = s, ~ \fa x \in\Omega, \label{eq:5.3.111a}\\
\phi(0) = &~ \phi_L \approx q(0), ~~ \phi(1)= \phi_R = q(1), \label{eq:5.3.111b}
\end{align}
\ese

\ni The eq. (\ref{eq:5.3.111a}) can be re-written as
\be \label{eq:5.3.2}
\big( b \phi - (\ep + \mu b) \phi^{\prime} \big)^{\prime} = s, ~ \fa x \in\Omega. \hspace{0.5cm}
\ee

\ni The flux corresponding to (\ref{eq:5.3.2}) is given by
\be \label{eq:5.3.3}
f = b \phi - (\ep + \mu b) \phi^{\prime}. \hspace{1.8cm}
\ee

\vspace{0.25cm}
\ni The Finite Volume-Complete Flux scheme is given as follows:

\vspace{0.25cm}
Perform uniform discretization on the domain $\Omega = (0,1)$ that leads to a uniform mesh. We assume the number of uniform mesh elements to be ${N-1}$ i.e., $N$ number of grid points, and thus ${N-1}$ number of interfaces. Therefore, we have
\begin{align*}
h =&~ \frac{1}{N-1}, \\
x_j =&~ (j-1) h,~ j = 1,2,...,N, \\
x_{j+{1/2}} =&~ \frac{1}{2} (x_j + x_{j+1}),~ j = 1,2,...,{N-1}
\end{align*}
where $h$ is the step size, $x_j$ grid points and $x_{j+{1/2}}$ interfaces. Also, we assume that $\Omega_j = (x_{j-{1/2}},x_{j+{1/2}})$ is the control volume. Now from (\ref{eq:5.3.2}) and (\ref{eq:5.3.3}), we have
\begin{align*}
f^{\prime} =~ s, ~ \forall x \in\Omega.
\end{align*}
Integrating on $\Omega_j$, we get
\begin{align*}
\int_{\Omega_j}&f^{\prime}\textrm{d}x = \int_{\Omega_j}s \hspace{0.04cm} \textrm{d}x,\hspace{2cm} \\
\Rightarrow~~f(x_{j+{1/2}})& - f(x_{j-{1/2}}) = \int_{\Omega_j}s \hspace{0.04cm} \textrm{d}x
\end{align*}
By Midpoint rule, the FVM reads for the above equation
\be \label{eq:5.3.4}
F_{j+{1/2}} - F_{j-{1/2}} = s_j \hspace{0.02cm} h, \hspace{1cm}
\ee
where $F_{j+{1/2}}$ and $F_{j-{1/2}}$ are the numerical fluxes at the cell interfaces $x_{j+{1/2}}$ and $x_{j-{1/2}}$, respectively and $s_j = s(x_j)$.

\vspace{0.5cm}
\ni\textbf{CF-scheme:}\\
\indent The derivation of the numerical flux $F_{j+{1/2}}$ is based on the following model BVP:
\bse \label{eq:5.3.5}
\begin{align}
& f^{\prime} = \big(b \phi - (\ep + \mu b) \phi^{\prime}\big)^{\prime} = ~s, ~~ x_j < x < x_{j+1}, \label{eq:5.3.5a}\\
& \phi(x_j) =  \phi_j, ~~ \phi(x_{j+1})= ~\phi_{j+1}. \label{eq:5.3.5b}
\end{align}
\ese
Now, we define the variables $\lambda$, $P$, $\Lambda$ and $S$ as follows
\begin{align} \label{eq:5.3.6}
\lambda \coloneqq \frac{b}{\ep + \mu b}, ~ P \coloneqq \lambda h, ~ \Lambda(x) \coloneqq \int_{x_{j+{1/2}}}^x \lambda(\xi) \textrm{d}\xi, ~ S(x)\coloneqq \int_{x_{j+{1/2}}}^x {s}(\xi)\textrm{d}\xi,
\end{align}
with $h = x_{j+1} - x_j$ where $P$ and $\Lambda$ are called Pe\'{c}let number and Pe\'{c}let integral, respectively. Integrating (\ref{eq:5.3.5a}) from $x_{j+{1/2}}$ to any point $x\in(x_j,x_{j+1})$, we get the integral balance
\be \label{eq:5.3.7}
f(x) - f_{j+{1/2}} = S(x), \hspace{0.5cm}
\ee
where $f_{j+{1/2}} = f(x_{j+{1/2}})$. By using the definition of $\Lambda$, The flux $f$ can be re-written as
\be \label{eq:5.3.8}
f = - (\ep + \mu b)(\phi e^{-\Lambda})^{\prime} e^{\Lambda}. \hspace{0.5cm}
\ee

\ni Now, substituting (\ref{eq:5.3.8}) in (\ref{eq:5.3.7}) and integrating from $x_j$ to $x_{j+1}$, we get the following expressions for the flux $f_{j+{1/2}}$ :
\bse \label{eq:5.3.9}
\begin{align}
f_{j+{1/2}} =& ~~~~ f^\textrm{h}_{j+{1/2}} + f^\textrm{i}_{j+{1/2}}, \label{eq:5.3.9a}\\
f^\textrm{h}_{j+{1/2}} =&~~~ \left. (\phi_j e^{-\Lambda_j} - \phi_{j+1} e^{-\Lambda_{j+1}})\middle/ \int_{x_j}^{x_{j+1}} {(\ep + \mu b)}^{-1} e^{-\Lambda} \textrm{d}x \right., \label{eq:5.3.9b}\\
f^\textrm{i}_{j+{1/2}} =&~ -\left.\int_{x_j}^{x_{j+1}} {(\ep + \mu b)}^{-1} e^{-\Lambda} S \textrm{d}x \middle/ \int_{x_j}^{x_{j+1}} {(\ep + \mu b)}^{-1} e^{-\Lambda} \textrm{d}x  \right., \label{eq:5.3.9c}
\end{align}
\ese
where $f^\textrm{h}_{j+{1/2}}$ and $f^\textrm{i}_{j+{1/2}}$ are the homogeneous and inhomogeneous part corresponding to the homogeneous and particular solution of (\ref{eq:5.3.5}), respectively, and $\Lambda_j = \Lambda(x_j)$.

Here, $\Lambda(x) = \lambda(x - x_{j+{1/2}})$ as $b$ is constant, and if $s$ is also a constant on the interval $[x_j,x_{j+1}]$ then $S(x) = s(x - x_{j+{1/2}})$. Further, putting these expressions of $\lambda$ and $S$ in (\ref{eq:5.3.9b}) and (\ref{eq:5.3.9c}) and simplifying further, we get
\bse \label{eq:5.3.10}
\begin{align}
f^\textrm{h}_{j+{1/2}} =&~ \frac{\ep + \mu b}{h} \big[B(-P) \phi_j - B(P) \phi_{j+1}\big], \label{eq:5.3.10a}\\
f^\textrm{i}_{j+{1/2}} =&~ \Big(\frac{1}{2} - W(P)\Big) s\hspace{0.03cm}h,\label{eq:5.3.10b}
\end{align}
\ese
where the functions $B$ and $W$ are defined as follows
\begin{align*}
B(z) \coloneqq \frac{z}{e^z - 1} ~~~~\text{and}~~~~  W(z) \coloneqq \frac{e^z -1-z}{z(e^z -1)} = \frac{1}{z}\big(1-B(z)\big),
\end{align*}
and called Bernoulli function \cite{spnr87} and Weight function, respectively. It is clear that the inhomogeneous flux $f^\textrm{i}_{j+{1/2}}$ is of importance when $\lvert P\rvert\ \gg 1$. These functions satisfy the following properties
\begin{align*}
B(-z) = z + B(z),~~0 \le W(z)\le 1, ~~\text{and}~~ W(-z) + W(z) = 1.
\end{align*}

\ni To show the dependency, we can write the homogeneous flux as follows
\begin{align} \label{eq:5.3.11}
f^\textrm{h}_{j+{1/2}} ~=&~ \mathscr{F}\big((\ep + \mu b)/h, P; \phi_j, \phi_{j+1}\big) \no \\
&=~ \alpha_{j+{1/2}}\big((\ep + \mu b)/h, P\big) \phi_j - \beta_{j+{1/2}}\big((\ep + \mu b)/h, P\big)\phi_{j+1}.
\end{align}

Also, we can generalise these fluxes for the case when $b$ and $s$ are variables, i.e., $b$ is a function of $x$ and $s$ is a function of $x$ or $\phi$, then $\lambda$ and $P$ will not be constants anymore. For this, we define the usual inner product as follows
\be \label{eq:5.3.12}
\langle a, b\rangle \coloneqq \int_{x_j}^{x_{j+1}} a(x)b(x)\textrm{d}x. \hspace{3.3cm}
\ee
Therefore, (\ref{eq:5.3.9}) becomes
\bse \label{eq:5.3.13}
\begin{align}
f_{j+{1/2}} =& ~~~~ f^\textrm{h}_{j+{1/2}} + f^\textrm{i}_{j+{1/2}}, \label{eq:5.3.13a}\\
f^\textrm{h}_{j+{1/2}} =& ~~~ \left. (\phi_j e^{-\Lambda_j} - \phi_{j+1} e^{-\Lambda_{j+1}})\middle/ \langle {(\ep + \mu b)}^{-1}, e^{-\Lambda}\rangle \right., \label{eq:5.3.13b}\\
f^\textrm{i}_{j+{1/2}} =& - \left. \langle {(\ep + \mu b)}^{-1} S, e^{-\Lambda}\rangle \middle/ \langle {(\ep + \mu b)}^{-1}, e^{-\Lambda}\rangle.  \right. \label{eq:5.3.13c}
\end{align}
\ese
Now, we also have the following relations
\begin{align*}
\Lambda_{j+{1/2}} =&~ \frac{1}{2}(\Lambda_j + \Lambda_{j+1}), ~~ \Lambda_{j+1} - \Lambda_j = \int_{x_j}^{x_{j+1}} \lambda(\xi) \textrm{d}\xi = \langle \lambda, 1\rangle, \\
\Lambda_j =&~  \frac{1}{2}(\Lambda_j + \Lambda_{j+1}) - \frac{1}{2}(\Lambda_{j+1} - \Lambda_j) = \Lambda_{j+{1/2}} - \langle \lambda, 1\rangle /2, \\
\Lambda_{j+1} =&~  \frac{1}{2}(\Lambda_j + \Lambda_{j+1}) + \frac{1}{2}(\Lambda_{j+1} - \Lambda_j) = \Lambda_{j+{1/2}} + \langle \lambda, 1\rangle /2.
\end{align*}
Therefore, on using above relations the homogeneous flux (\ref{eq:5.3.13b}) becomes
\be \label{eq:5.3.14}
f^\textrm{h}_{j+{1/2}} =~~ \left. e^{-\Lambda_{j+{1/2}}}(e^{\langle \lambda, 1\rangle /2}\phi_j - e^{-\langle \lambda, 1\rangle /2}\phi_{j+1})\middle/ \langle {(\ep + \mu b)}^{-1}, e^{-\Lambda}\rangle. \right.
\ee
Now, we have
\begin{align*}
\langle \lambda, e^{-\Lambda}\rangle =&~ \int_{x_j}^{x_{j+1}} \lambda e^{-\Lambda} \textrm{d}x = [-e^{-\Lambda}]_{x_j}^{x_{j+1}} = e^{-\Lambda_j} - e^{-\Lambda_{j+1}} \\
=&~~ e^{-\Lambda_{j+{1/2}}}(e^{\langle \lambda, 1\rangle /2} - e^{-\langle \lambda, 1\rangle /2}), \\
\Rightarrow ~~~~ e^{-\Lambda_{j+{1/2}}} =&~ \left. \langle \lambda, e^{-\Lambda}\rangle \middle/ (e^{\langle \lambda, 1\rangle /2} - e^{-\langle \lambda, 1\rangle /2}). \right.
\end{align*}
Putting the above value of $e^{-\Lambda_{j+{1/2}}}$ in (\ref{eq:5.3.14}) and on some further simplification, we get
\be \label{eq:5.3.15}
f^\textrm{h}_{j+{1/2}} =~~ \frac{\langle \lambda, e^{-\Lambda}\rangle / \langle \lambda, 1\rangle}{\langle {(\ep + \mu b)}^{-1}, e^{-\Lambda}\rangle} \big[ B(-\langle \lambda, 1\rangle)\phi_j - B(\langle \lambda, 1\rangle)\phi_{j+1} \big].
\ee
This can also be written as a modification of the constant coefficient homogeneous flux (\ref{eq:5.3.11}) as follows
\be \label{eq:5.3.16}
f^\textrm{h}_{j+{1/2}} = \mathscr{F}\left(\frac{\langle \lambda, e^{-\Lambda}\rangle / \langle \lambda, 1\rangle}{\langle {(\ep + \mu b)}^{-1}, e^{-\Lambda}\rangle}, \langle \lambda, 1\rangle; \phi_j, \phi_{j+1}\right).
\ee

\ni Now, we will simplify the numerator part of the inhomogeneous flux (\ref{eq:5.3.13c}).
\begin{align*}
\langle {(\ep + \mu b)}^{-1} S, e^{-\Lambda}\rangle = & \int_{x_j}^{x_{j+1}} {(\ep + \mu b)}^{-1} e^{-\Lambda}S \textrm{d}x \\
= \int_{x_j}^{x_{j+{1/2}}} &{(\ep + \mu b)}^{-1} e^{-\Lambda}S \textrm{d}x + \int_{x_{j+{1/2}}}^{x_{j+1}} {(\ep + \mu b)}^{-1} e^{-\Lambda}S \textrm{d}x \\
= \int_{x_j}^{x_{j+{1/2}}} &{(\ep + \mu b)}^{-1} e^{-\Lambda} {\int_{x_{j+{1/2}}}^x s(\xi)\textrm{d}\xi} \textrm{d}x \\
&\hspace{0.6cm}+ \int_{x_{j+{1/2}}}^{x_{j+1}} {(\ep + \mu b)}^{-1} e^{-\Lambda} {\int_{x_{j+{1/2}}}^x s(\xi)\textrm{d}\xi} \textrm{d}x .
\end{align*}
We define normalised coordinates as follows
\begin{align*}
\sigma \coloneqq&~ \frac{x - x_j}{h}, ~~ (0 \le \sigma \le 1), \\
x =&~ x_j + h \sigma, ~~ \textrm{d}x = h \textrm{d} \sigma,\\
\xi =&~ x_j + h \eta, ~~ \textrm{d}\xi = h \textrm{d} \eta.
\end{align*}
Then, we have
\begin{align*}
\langle {(\ep + \mu b)}^{-1} S, e^{-\Lambda}\rangle =&~ {h}^2 \int_{0}^{1/2} {(\ep + \mu b)}^{-1} e^{-\Lambda} {\int_{1/2}^{\sigma} s(\eta)\textrm{d}\eta} \textrm{d}\sigma \\
& \hspace{2cm}+ {h}^2 \int_{1/2}^{1} {(\ep + \mu b)}^{-1} e^{-\Lambda}{\int_{1/2}^{\sigma} s(\eta)\textrm{d}\eta} \textrm{d}\sigma \\
=&~ {h}^2 \int_{0}^{1/2} \int_{\eta}^{0} {(\ep + \mu b)}^{-1} e^{-\Lambda}\textrm{d}\sigma s(\eta)\textrm{d}\eta \\
& \hspace{2cm} + {h}^2 \int_{1/2}^{1} \int_{\eta}^{1} {(\ep + \mu b)}^{-1} e^{-\Lambda}\textrm{d}\sigma s(\eta)\textrm{d}\eta \\
& \hspace{2cm} \text{(on changing the order of integration)}
\end{align*}
Therefore, the inhomogeneous flux (\ref{eq:5.3.13c}) can be written as
\be \label{eq:5.3.17}
f^\textrm{i}_{j+{1/2}} = h \int_{0}^{1} G(\eta) s(\eta)\textrm{d}\eta, \hspace{2cm}
\ee
where $G(\eta)$ is the Green's function for the flux defined as
\begin{align} \label{eq:5.3.18}
G(\eta) = \begin{cases}
             {h} \int_{0}^{\eta} {(\ep + \mu b)}^{-1} e^{-\Lambda}\textrm{d}\sigma / \langle {(\ep + \mu b)}^{-1}, e^{-\Lambda}\rangle,~~~~ 0 \le \eta \le \frac{1}{2}, \\
             -{h} \int_{\eta}^{1} {(\ep + \mu b)}^{-1} e^{-\Lambda}\textrm{d}\sigma / \langle {(\ep + \mu b)}^{-1}, e^{-\Lambda}\rangle,~~ \frac{1}{2} \le \eta \le 1.
            \end{cases}
\end{align}
When $b = \text{Const} \ne 0$ is a constant, we have
\begin{align*}
{h} \int_{0}^{\eta} {(\ep + \mu b)}^{-1} e^{-\Lambda}\textrm{d}\sigma ~~=&~~ \int_{x_j}^{\xi} {(\ep + \mu b)}^{-1} e^{-\Lambda} \textrm{d}x = \frac{1}{b}\int_{x_j}^{\xi} \lambda e^{-\Lambda}\textrm{d}x \\
=&~~ \frac{1}{b} [-e^{-\Lambda}]_{x_j}^{\xi} =~~ \frac{1}{b} [e^{-\Lambda_j} - e^{-\Lambda(\xi)}],
\intertext{and}
\langle {(\ep + \mu b)}^{-1}, e^{-\Lambda}\rangle ~~=&~~ \frac{1}{b} \langle \lambda, e^{-\Lambda}\rangle \\
=&~~ \frac{1}{b} \int_{x_j}^{x_{j+1}} \lambda, e^{-\Lambda} \textrm{d}x = \frac{1}{b} [-e^{-\Lambda}]_{x_j}^{x_{j+1}} \\
=&~~ \frac{1}{b} [e^{-\Lambda_j} - e^{-\Lambda_{j+1}}].
\end{align*}
Now, for $0 \le \eta \le \frac{1}{2}$, we have
\begin{align*}
G(\eta) = \frac{e^{-\Lambda_j} - e^{-\Lambda(\xi)}}{e^{-\Lambda_j} - e^{-\Lambda_{j+1}}} = \frac{1-e^{\Lambda_j - \Lambda(\xi)}}{1-e^{-\langle \lambda, 1\rangle}}.
\end{align*}
Here, we have
\begin{align*}
\Lambda_j - \Lambda(\xi) =&~ \int_{x_{j+{1/2}}}^{x_j} \lambda(x)\textrm{d}x - \int_{x_{j+{1/2}}}^{\xi} \lambda(x)\textrm{d}x \\
=&~ - \int_{x_j}^{\xi} \lambda(x)\textrm{d}x = - \sigma \langle \lambda, 1\rangle
\end{align*}
\hspace{1cm}(where $\sigma(x) \coloneqq ~ \int_{x_j}^{x} \lambda(\xi)\textrm{d}\xi / \langle \lambda, 1\rangle$,~ a weighted normalised coordinate)

\ni Therefore, for the case for $0 \le \eta \le \frac{1}{2}$, we have
\begin{align*}
G(\sigma ; \langle \lambda, 1\rangle) = \frac{1-e^{-\langle \lambda, 1\rangle \sigma}}{1-e^{-\langle \lambda, 1\rangle}} = \frac{1-e^{-P \sigma}}{1-e^{-P}}, \big( \because \langle \lambda, 1\rangle = P,~\text{when}~ b = \text{Const} \ne 0 \big)
\end{align*}
Similarly, for the case for $\frac{1}{2} \le \eta \le 1$, we have
\begin{align*}
G(\sigma ; \langle \lambda, 1\rangle) = - \frac{1-e^{\langle \lambda, 1\rangle (1-\sigma)}}{1-e^{\langle \lambda, 1\rangle}} = - \frac{1-e^{P (1-\sigma)}}{1-e^{P}}.
\end{align*}
Then, the Green's function for the flux becomes (when $b = \text{Const} \ne 0$)
\begin{align} \label{eq:5.3.19}
G(\sigma ; P) = \begin{cases}
               ~~\frac{1-e^{-P \sigma}}{1-e^{-P}},~~~~~~~ 0 \le \sigma \le \frac{1}{2}, \\
               - \frac{1-e^{P (1-\sigma)}}{1-e^{P}},~~ \frac{1}{2} \le \sigma \le 1.
                \end{cases}
\end{align}

This Green's function $G$ for the flux is different from the usual Green's function as this Green's function $G$ relates the flux to the source term, while, the usual Green's function relates the solution to the source term. This Green's function $G$ is discontinuous at $\sigma = \frac{1}{2}$, corresponding to $x = x_{j+1/2}$, having jump $G(\frac{1}{2}-;P) - G(\frac{1}{2}+;P) = 1$. For more details see \cite{kwm96,bnkmp11}.

\vspace{0.2cm}
\ni Thus, the inhomogeneous flux (\ref{eq:5.3.17}) can be written as (when $b = \text{Const} \ne 0$)
\be \label{eq:5.3.20}
f^\textrm{i}_{j+{1/2}} = h \int_{0}^{1} G(\sigma ; \langle \lambda, 1\rangle) s(\sigma)\textrm{d}\sigma,
\ee
with $G(\sigma;P)$ the constant coefficient Green's function given in (\ref{eq:5.3.19}).

\vspace{0.2cm}
Now, we need to approximate $\langle \lambda, 1\rangle$, $\langle d, e^{-\Lambda}\rangle$, $(d = \lambda, {(\ep + \mu b)}^{-1})$ and the integral $\int_{0}^{1} G(\sigma ;P) s(\sigma) \textrm{d}\sigma$. For this purpose, we give quadrature rules for the inner products and an approximation for the integration in (\ref{eq:5.3.20}). We introduce the following
\bse \label{eq:5.3.21}
\begin{align}
&\langle \lambda, 1\rangle ~~~~ \approx ~~h \frac{1}{2}(\lambda_j + \lambda_{j+1}) \coloneqq h \bar{\lambda}_{j+{1/2}} = \bar{P}_{j+{1/2}}, \label{eq:5.3.21a}\\
&\hspace{7cm}~(\text{Trapezoidal Rule}) \no \\
&\frac{\langle d, e^{-\Lambda}\rangle}{\langle 1, e^{-\Lambda}\rangle} ~~~ \cong~~ \tilde{d}_{j+{1/2}} \coloneqq W(-\bar{P}_{j+{1/2}})d_j + W(\bar{P}_{j+{1/2}})d_{j+1}, \label{eq:5.3.21b}\\
&\hspace{7cm}~(\text{Weighted Average}) \no \\
&\int_{0}^{1} G(\sigma; \langle \lambda, 1\rangle) s(\sigma) \textrm{d}\sigma \coloneqq~ \int_{0}^{1} G(\sigma; \langle \lambda, 1\rangle) s_{b,j+{1/2}} \textrm{d}\sigma \no \\
& \hspace{3.5cm}= \Big(\frac{1}{2}- W(\bar{P}_{j+{1/2}})\Big) s_{b,j+{1/2}}, \label{eq:5.3.21c} \\
&s_{b,j+{1/2}} \coloneqq ~ \begin{cases} \label{eq:5.3.21d}
                                                     s_j,~~~~\text{if}~~\bar{b}_{j+{1/2}} \geq0,\\
                                                     s_{j+1},~\text{if}~~\bar{b}_{j+{1/2}} <0.
                                                     \end{cases}
\end{align}
\ese
where $\bar{\lambda}_{j+{1/2}}$ is the average, $\tilde{d}_{j+{1/2}}$ weighted average, $W$ weight function as defined before and
$s_{b,j+{1/2}}$ the upwind value of $s(\sigma)$. Therefore, from (\ref{eq:5.3.15}), we have
\begin{align*}
F^\textrm{h}_{j+{1/2}} =~ \frac{\langle \lambda, e^{-\Lambda}\rangle / \langle \lambda, 1\rangle}{\langle {(\ep + \mu b)}^{-1}, e^{-\Lambda}\rangle} \Big[ B(-\bar{P}_{j+{1/2}})\phi_j - B(\bar{P}_{j+{1/2}})\phi_{j+1} \Big]
\end{align*}

\ni Here, we see that
\begin{align*}
\frac{\langle \lambda, e^{-\Lambda}\rangle / \langle \lambda, 1\rangle}{\langle {(\ep + \mu b)}^{-1}, e^{-\Lambda}\rangle} =&~
\frac{\langle \lambda, e^{-\Lambda}\rangle / \langle 1, e^{-\Lambda}\rangle}{\langle {(\ep + \mu b)}^{-1}, e^{-\Lambda}\rangle /
\langle 1, e^{-\Lambda}\rangle}~ \frac{1}{\langle \lambda, 1\rangle} \\
=&~ \frac{\tilde{\lambda}_{j+{1/2}}}{\widetilde{(\ep + \mu b)^{-1}_{j+{1/2}}}}~ \frac{1}{h \bar{\lambda}_{j+{1/2}}} \\
=&~ \frac{1}{h}~ \frac{\tilde{\lambda}_{j+{1/2}}}{\bar{\lambda}_{j+{1/2}}}~ \widetilde{(\ep + \mu b)}_{j+{1/2}}.
\end{align*}
Thus, the homogeneous flux becomes
\begin{align} \label{eq:5.3.22}
F^\textrm{h}_{j+{1/2}} =&~ \frac{\varepsilon_{j+{1/2}}}{h} \left[ B(-\bar{P}_{j+{1/2}})\phi_j - B(\bar{P}_{j+{1/2}})\phi_{j+1} \right], \\
&\hspace{2cm} \text{ where }~~\varepsilon_{j+{1/2}} = \frac{\tilde{\lambda}_{j+{1/2}}}{\bar{\lambda}_{j+{1/2}}}~ \widetilde{(\ep + \mu b)}_{j+{1/2}}. \no
\end{align}

\ni Here, we notice that $\widetilde{(\ep + \mu b)}_{j+{1/2}} = \ep + \mu b $ when $\ep$, $\mu$ and $b$ are constants, also, then $\bar{\lambda}_{j+{1/2}} = \tilde{\lambda}_{j+{1/2}} = \lambda$ and $\varepsilon_{j+{1/2}} = \ep + \mu b $. Further, from (\ref{eq:5.3.20}), the inhomogeneous numerical flux becomes
\begin{align} \label{eq:5.3.23}
F^\textrm{i}_{j+{1/2}} = \Big(\frac{1}{2}- W(\bar{P}_{j+{1/2}})\Big) s_{b,j+{1/2}} h. \hspace{2cm}
\end{align}
Therefore, the final numerical flux \cite{hof98} at the cell interface $x_{j+{1/2}}$ is given by
\begin{align} \label{eq:5.3.24}
F_{j+{1/2}} =&~~ F^\textrm{h}_{j+{1/2}} + F^\textrm{i}_{j+{1/2}}, \no \\
\Rightarrow ~F_{j+{1/2}} =&~~ \frac{\varepsilon_{j+{1/2}}}{h} \big[ B(-\bar{P}_{j+{1/2}})\phi_j - B(\bar{P}_{j+{1/2}})\phi_{j+1} \big] \no \\
&\hspace{3cm}+ \Big(\frac{1}{2}- W(\bar{P}_{j+{1/2}})\Big) s_{b,j+{1/2}} h.
\end{align}
Now, we define some coefficients (see \cite{bnkmp11} by Ten Thije Boonkkamp) as follows
\bse \label{eq:5.3.25}
\begin{align}
\alpha_{j+{1/2}} \coloneqq&~ \frac{\varepsilon_{j+{1/2}}}{h} B(-\bar{P}_{j+{1/2}}),~~ \beta_{j+{1/2}} \coloneqq \frac{\varepsilon_{j+{1/2}}}{h} B(\bar{P}_{j+{1/2}}), \label{eq:5.3.25a} \\
\gamma_{j+{1/2}} \coloneqq&~ \max\Big(\frac{1}{2}- W(\bar{P}_{j+{1/2}}),0\Big),~ \delta_{j+{1/2}} \coloneqq \min\Big(\frac{1}{2}- W(\bar{P}_{j+{1/2}}),0\Big). \label{eq:5.3.25b}
\end{align}
\ese
Further, by using (\ref{eq:5.3.25}) in (\ref{eq:5.3.24}), the final numerical flux becomes
\begin{align} \label{eq:5.3.26}
F_{j+{1/2}} =&~\alpha_{j+{1/2}} \phi_j - \beta_{j+{1/2}} \phi_{j+1}+  h \left( \gamma_{j+{1/2}}s_j + \delta_{j+{1/2}}s_{j+1} \right).
\end{align}
Likewise, we have
\begin{align} \label{eq:5.3.27}
F_{j-{1/2}} =&~\alpha_{j-{1/2}} \phi_{j-1}- \beta_{j-{1/2}} \phi_j +  h \left( \gamma_{j-{1/2}}s_{j-1} + \delta_{j-{1/2}} s_j \right),
\end{align}
where, for $b>0$, $\delta_{j\pm{1/2}}= 0$, $\gamma_{j\pm{1/2}} \neq 0$, and for $b<0$, $\delta_{j\pm{1/2}} \neq 0$, $\gamma_{j\pm{1/2}}= 0$.

\ni Putting the values of $F_{j+{1/2}}$ and $F_{j-{1/2}}$ in (\ref{eq:5.3.4}), we get
\begin{align} \label{eq:5.3.28}
- \alpha_{j-{1/2}} \phi_{j-1} ~+&~ (\alpha_{j+{1/2}} + \beta_{j-{1/2}}) \phi_j - \beta_{j+{1/2}} \phi_{j+1} \no \\
=&~ h \big[\gamma_{j-{1/2}} s_{j-1} + (1 - \gamma_{j+{1/2}} + \delta_{j-{1/2}}) s_j - \delta_{j+{1/2}} s_{j+1}\big],
\end{align}
which is the \emph{complete flux scheme} (CFS). In the above equation (\ref{eq:5.3.28}), when $b$ is constant, all the coefficients will also be constants.

Note that the FV-CF scheme has a three-point coupling for both $\phi$ and $s$, resulting into the following linear system
\be \label{eq:5.3.29}
\boldsymbol{A} \boldsymbol{\phi} = \boldsymbol{B} \boldsymbol{s} + \boldsymbol{b},
\ee
where $\boldsymbol{A}$, $\boldsymbol{B} \in \mathbb{R}^{(N-1)\times(N-1)}$ are tridiagonal matrices, $\boldsymbol{\phi}$ and $\boldsymbol{s}$ the vectors of unknowns and source terms, respectively, and $\boldsymbol{b}$ a vector containing boundary data.

Now, we also consider the special case when $b = 0$, in this case, (\ref{eq:5.3.2}) takes the form $-(\ep \phi^\prime)^\prime = s$, and we have $\bar{P}_{j\pm{1/2}} = 0$, and consequently $F^\textrm{i}_{j+{1/2}}$ vanishes, leads us the second order central difference scheme
\begin{align} \label{eq:5.3.30}
- \frac{1}{h} \big[ \bar{\ep}_{j+{1/2}}(\phi_{j+1} - \phi_j) - \bar{\ep}_{j-{1/2}}(\phi_j - \phi_{j-1}) \big] =&~ s_j h.
\end{align}
where $\bar{\ep}_{j+{1/2}} = \bar{\ep}_{j-{1/2}} = \ep$, as $\ep$ is constant.

\section{Stability}
\label{sec:5.4}
The FV-CF scheme (\ref{eq:5.3.28}) can be written as follows
\be \label{eq:5.4.1}
\mathscr{L}^{h}\phi_j = \mathscr{W}^{h}s_j,~~j = 1,2,...,{N-1}, \hspace{0.8cm}
\ee
where the difference operator $\mathscr{L}^{h}$ and the weighting operator $\mathscr{W}^{h}$ are defined as follows
\be \label{eq:5.4.2}
\mathscr{L}^{h} \phi_j \coloneqq -a_{W,j} \phi_{j-1} + a_{C,j} \phi_j - a_{E,j} \phi_{j+1},
\ee
\be \label{eq:5.4.3}
\mathscr{W}^{h} s_j \coloneqq~ b_{W,j} s_{j-1}+ b_{C,j} s_j + b_{E,j} s_{j+1},~~
\ee
with coefficients defined as follows
\begin{align} \label{eq:5.4.4}
a_{W,j} &\coloneqq~ \frac{1}{h} \alpha_{j-{1/2}},~~ a_{E,j} \coloneqq~ \frac{1}{h} \beta_{j+{1/2}},~~ a_{C,j} \coloneqq~ \frac{1}{h} (\alpha_{j+{1/2}} + \beta_{j-{1/2}}), \no \\
b_{W,j} &\coloneqq~ \gamma_{j-{1/2}},~~ b_{E,j} \coloneqq -\delta_{j+{1/2}},~~ b_{C,j} \coloneqq~ 1- \gamma_{j+{1/2}}+ \delta_{j-{1/2}}.
\end{align}

Now, for the BVP (\ref{eq:5.3.1}), we assume $0 < \ep \ll 1$, $\mu > 0$ and $b>0$ are constants, and $s\in C^m[0,1]$.
Therefore, the coefficients defined in (\ref{eq:5.4.4}) are constants and take the forms as follows
\begin{align} \label{eq:5.4.5}
a_{W} &\coloneqq~ (\ep + \mu b) h^{-2} B(-P),~~ a_{E} \coloneqq~ (\ep + \mu b) h^{-2} B(P),~~ a_{C} \coloneqq~ a_{W} + a_{E}, \no \\
b_{W} &\coloneqq~ \frac{1}{2} - W(P),~~ b_{E} \coloneqq 0,~~ b_{C} \coloneqq~ \frac{1}{2} + W(P).
\end{align}

\ni Here, we see that all these coefficients are positive, and in system of equation (\ref{eq:5.3.29}), $\boldsymbol{A}$ is a tridiagonal matrix and $\boldsymbol{B}$ a lower bi-diagonal matrix. Further, we have assumed $C,~c,~c_1$ and $c_2$ to be positive constants, independent of $\ep,~\mu,~b$ and $h$, in the remainder of this paper.

\begin{lemma} \label{lemma:5.1}
The linear system (\ref{eq:5.3.29}) has a unique solution. If $\mathscr{L}^{h}\phi_j \leq \mathscr{L}^{h}\psi_j$, $j = 1,2,...,{N-1}$, and if $\phi_0 \leq \psi_0,~\phi_N \leq \psi_N$, then $\phi_j \leq \psi_j,j = 1,2,...,{N-1}$.
\end{lemma}
\begin{proof}
This lemma can be proved from the fact that $\boldsymbol{A} = (a_{ij})$ is an $M$-matrix. In fact we notice that $\boldsymbol{A}$
is irreducibly  diagonally dominant, i.e., $\boldsymbol{A}$ is irreducible, and $\lvert a_{ii}\rvert \geq \sum_{j\neq i} \lvert a_{ij}\rvert$ with strict inequality for at least one row. Thus, $\boldsymbol{A}$ has a positive inverse, i.e. $\boldsymbol{A}$ is
of monotone type. This shows that the system (\ref{eq:5.3.29}) has a unique solution.

Moreover, if $(\boldsymbol{A}\phi)_j = \mathscr{L}^{h}\psi_j,j = 1,2,...,{N-1}$; then the monotonicity of $\boldsymbol{A}$ implies the monotonicity of $\mathscr{L}^{h}$. Now, for the points next to the boundaries, e.g., $x_1$
\begin{align*}
\mathscr{L}^{h}\phi_1 \leq \mathscr{L}^{h}\psi_1 \Rightarrow ~~ &a_C(\phi_1 - \psi_1) - a_E(\phi_2 - \psi_2) \leq a_W(\phi_0 - \psi_0) \leq 0, \hspace{0.5cm}\\
\Rightarrow  ~~ &(\boldsymbol{A \phi})_1 \leq (\boldsymbol{A \psi})_1.
\end{align*}
\ni Likewise, we have
\begin{align*}
(\boldsymbol{A \phi})_{N-1} \leq (\boldsymbol{A \psi})_{N-1}.\hspace{2cm}
\end{align*}
Therefore, $\mathscr{L}^{h}\phi_j \leq \mathscr{L}^{h}\psi_j$ implies that $\phi_j \leq \psi_j,j = 1,2,...,{N-1}$ under the conditions
of the lemma. This proves the lemma.
\end{proof}

From the system (\ref{eq:5.3.29}), we have the following relation
\be \label{eq:5.4.6}
\boldsymbol{e} = \boldsymbol{A^{-1}} \boldsymbol{\uptau},
\ee
where $\boldsymbol{e}$ and $\boldsymbol{\uptau}$ are the discretization error and truncation error, respectively.

\ni If $\lVert \boldsymbol{A^{-1}} \rVert_{\infty}$ is bounded, then the CFS is stable.

\begin{lemma} \label{lemma:5.2}
\cite{kbnr03,likmp13} There exists a constant $C > 0$, such that
\begin{align} \label{eq:5.4.7}
\lVert \boldsymbol{A^{-1}} \rVert_{\infty} \leq - \frac{1}{b} \left( \frac{1}{J} \ln B(J) + W(J) \right) \leq C,
\end{align}

\ni where $J = b/ {(\ep + \mu b)}$. Thus, $\lVert \boldsymbol{A^{-1}} \rVert_{\infty}$ is bounded.
\end{lemma}
Thus, $\lVert \boldsymbol{A^{-1}} \rVert_{\infty}$ bounded implies that CFS is stable.

\section{Consistency}
\label{sec:5.5}
The truncation error for the CFS is defined as follows
\be \label{eq:5.5.1}
\uptau_j \coloneqq~ \mathscr{L}^{h}\phi(x_j) - \mathscr{W}^{h}(\mathscr{L}\phi)(x_j),~~j = 1,2,...,{N-1}.
\ee
We find the expression for $\uptau_j$ when $h \leq \ep + \mu b $ and $h \geq \ep + \mu b $. For the same, we use the Taylor expansion
\bse \label{eq:5.5.2}
\begin{align}
f(x_2) ~~=&~~ \sum_{n=0}^q \frac{f^{(n)}(x_1)}{n!} (x_2-x_1)^n + R_q(x_1,x_2;f), \label{eq:5.5.2a}\\
\intertext{where, $R_q(x_1,x_2;f)$ is the remainder term, given by}
R_q(x_1,x_2;f) =&~ \frac{1}{q!} \int_{x_1}^{x_2} (x_2-x)^q f^{(q+1)}(x) \textrm{d}x, \label{eq:5.5.2b}
\end{align}
\ese
with $f(x)$ to be smooth enough.

\vspace{0.5cm}
\ni \textbf{Case 1}: When $h \leq \ep + \mu b $

\vspace{0.25cm}
Using the Taylor expansion (\ref{eq:5.5.2a}) up to the fourth derivative of $\phi$, from (\ref{eq:5.5.1}) we have
\begin{align} \label{eq:5.5.3}
\uptau_j =&~ T_3 \phi^{\prime\prime\prime}(x_j) + I_1 + I_2 + I_3 + I_4, \\
\intertext{where}
T_3 =&~ \frac{b h^2}{6} - \Big(\frac{1}{2} - W(P)\Big) \Big((\ep + \mu b) h + \frac{b h^2}{2}\Big), \no \\
I_1 =& - (\ep + \mu b) h^{-2} B(-P) R_3(x_j,x_{j}-h; \phi), \no \\
I_2 =& - (\ep + \mu b) h^{-2} B(P) R_3(x_j,x_{j}+h; \phi), \no \\
I_3 =& - b \Big(\frac{1}{2} - W(P)\Big) R_2(x_j,x_{j}-h; \phi^{\prime}), \no \\
I_4 =&~ (\ep + \mu b) \Big(\frac{1}{2} - W(P)\Big) R_1(x_j,x_{j}-h; \phi^{\prime\prime}). \no
\end{align}
as $T_1$ and $T_2$, coefficients of $\phi^{\prime}$ and $\phi^{\prime\prime}$, respectively, vanish.

\vspace{0.5cm}
\ni \textbf{Case 2}: When $h \geq \ep + \mu b $

\vspace{0.25cm}
In this case, $ \ep + \mu b $ contributes one order in the estimation, so that it suffices to use the Taylor expansion (\ref{eq:5.5.2}) up to the third derivative of $\phi$. Therefore, we have
\begin{align} \label{eq:5.5.4}
\uptau_j =&~ I_1 + I_2 + I_3 + I_4, \\
\intertext{where}
I_1 =& - (\ep + \mu b) h^{-2} B(-P) R_2(x_j,x_{j}-h; \phi), \no \\
I_2 =& - (\ep + \mu b) h^{-2} B(P) R_2(x_j,x_{j}+h; \phi), \no \\
I_3 =& - b \Big(\frac{1}{2} - W(P)\Big) R_1(x_j,x_{j}-h; \phi^{\prime}), \no \\
I_4 =&~ (\ep + \mu b) \Big(\frac{1}{2} - W(P)\Big) R_0(x_j,x_{j}-h; \phi^{\prime\prime}). \no
\end{align}

Now, if the derivatives of the solution $\phi(x)$ are uniformly bounded, then from (\ref{eq:5.5.4}) and (\ref{eq:5.5.5}), we can directly establish the following lemma.
\begin{lemma} \label{lemma:5.3}
Let $\phi(x)$ be the solution of (\ref{eq:5.3.1}). If the first four derivatives of $\phi(x)$ are uniformly bounded, then we have
\begin{align} \label{eq:5.5.5}
\lvert \uptau_j \rvert \leq C h^2.
\end{align}
\end{lemma}

By using this lemma along with Lemma \ref{lemma:5.2}, the second-order convergence of CFS can be established. But when the derivatives are not bounded i.e., when an inner or boundary layer exists, then we need the following lemma to bound the derivatives.
\begin{lemma} \label{lemma:5.4}
\cite{kets} The solution $\phi(x)$ of (\ref{eq:5.3.1}) can be decomposed as
\bse \label{eq:5.5.6}
\begin{align}
\phi(x) =&~ r y(x) + z(x), \label{eq:5.5.6a}
\intertext{where $\lvert r \rvert \leq c_1$ and $y(x) = \exp(-b{(\ep + \mu b)}^{-1}(1-x))$, and}
\lvert z^{(i)}(x)\rvert \leq &~ c_2 \Big( 1+ {(\ep + \mu b)}^{-i+1}\exp(-b{(\ep + \mu b)}^{-1}(1-x)) \Big), \label{eq:5.5.6b}
\end{align}
with $c_1 > 0$ and $c_2 >0$, independent of $(\ep + \mu b)$.
\ese
\end{lemma}

From the above Lemma \ref{lemma:5.4}, the solution $\phi(x)$ of (\ref{eq:5.3.1}) can be decomposed into two terms. The first term $y(x)$ of this decomposition is the solution of the homogeneous equation of (\ref{eq:5.3.1}) that can be easily verified. Also, the CFS is exact for the constant coefficient homogeneous problem, i.e., from the first term, truncation error is zero. Then, the truncation error takes the from
\be \label{eq:5.5.7}
\uptau_j \coloneqq~ \mathscr{L}^{h}z(x_j) - \mathscr{W}^{h}(\mathscr{L}z)(x_j),~~j = 1,2,...,{N-1}.
\ee

Now, we estimate the terms in truncation error in both the cases, i.e., when $h \leq \ep + \mu b $ and $h \geq \ep + \mu b $. For the first case when $h \leq \ep + \mu b $, the truncation error is given by (\ref{eq:5.5.3}). Then we have
\begin{align*}
\lvert T_3 z^{\prime\prime\prime}(x_j) \rvert =&~ \left\lvert \Big[\frac{b h^2}{6} - \Big(\frac{1}{2} - W(P)\Big)\frac{b h^2}{2}
- \Big(\frac{1}{2} - W(P)\Big) (\ep + \mu b) h \Big] z^{\prime\prime\prime}(x_j) \right\rvert \\
\leq &~ \left\lvert \frac{b h^2}{6} - \Big(\frac{1}{2} - W(P)\Big)\frac{b h^2}{2}
- \Big(\frac{1}{2} - W(P)\Big) (\ep + \mu b) h \right\rvert \lvert z^{\prime\prime\prime}(x_j) \rvert \\
\leq &~ \Big[ \frac{b h^2}{6} + \Big(\frac{1}{2} - W(P)\Big)\frac{b h^2}{2}
+ \Big(\frac{1}{2} - W(P)\Big) (\ep + \mu b) h \Big] \lvert z^{\prime\prime\prime}(x_j) \rvert
\end{align*}
\begin{align*}
\leq &~ \Big[ \frac{b h^2}{6} + \frac{b h^2}{4} + \Big(\frac{1}{2} - W(P)\Big) \frac{b h^2}{P} \Big] \lvert z^{\prime\prime\prime}(x_j) \rvert \\
\leq &~ \Big[ \frac{b h^2}{6}+ \frac{b h^2}{4}+ \frac{1}{12}{b h^2}+ O(h^4) \Big] \lvert z^{\prime\prime\prime}(x_j) \rvert \\
\leq &~ C h^2 \big(1 + (\ep + \mu b)^{-2} y(x_j)\big) \hspace{3.2cm} \big( \text{using}~ (\ref{eq:5.5.6b}) \big)
\end{align*}

Now for the remainder terms, because of the similarity, only the estimation for $I_1 + I_2$ is presented, and for this estimation, the following relation is used
\be \label{eq:5.5.8}
\frac{(\ep + \mu b)}{2}\big(B(-P) + B(P)\big) = b h \Big(\frac{1}{2} - W(P)\Big) + (\ep + \mu b)
\ee
Therefore, we have
\begin{align*}
&~~\lvert I_1 + I_2 \rvert \\
&~\leq (\ep + \mu b) h^{-2} B(-P) \frac{1}{6} \left\lvert \int_{x_j}^{x_j- h} (x_j- h- t)^{3} z^{(4)}(t)\textrm{d}t \right\rvert \\
&~\hspace{3cm}+(\ep+\mu b) h^{-2} B(P)\frac{1}{6} \left\lvert \int_{x_j}^{x_j+ h} (x_j +h - t)^{3} z^{(4)}(t)\textrm{d}t \right\rvert \\
&~ \leq \left[\frac{(\ep + \mu b) h}{6}\left(B(-P) + B(P)\right)\right] \int_{x_j- h}^{x_j+ h} \lvert z^{(4)}(t) \rvert \textrm{d}t \\
&~ \leq \frac{h}{3}\left[b h \left(\frac{1}{2}- W(P)\right)+ (\ep+\mu b) \right] \int_{x_j- h}^{x_j+ h} \lvert z^{(4)}(t) \rvert \textrm{d}t \hspace{1.5cm} \big( \text{using}~ (\ref{eq:5.5.8}) \big) \\
&~ \leq C\big((\ep + \mu b) h + h^2\big) \int_{x_j- h}^{x_j+ h} \lvert z^{(4)}(t) \rvert \textrm{d}t \\
&~ \leq C\big((\ep + \mu b) h + h^2\big) \int_{x_j- h}^{x_j+ h} \big( 1+ {(\ep + \mu b)}^{-3}\exp(-b{(\ep + \mu b)}^{-1}(1-t)) \big) \textrm{d}t \\
&~ \leq C\big((\ep + \mu b) h + h^2\big) \Big[ h+ {(\ep + \mu b)}^{-2} b^{-1}\sinh(b h (\ep + \mu b)^{-1})\\
&~ \hspace{7cm}  \exp(-b (\ep + \mu b)^{-1}(1-x_j)) \Big]
\end{align*}
\begin{align*}
&~ \leq C\big((\ep + \mu b) h + h^2\big) \Big[ h+ {(\ep + \mu b)}^{-3} h \exp(-b (\ep + \mu b)^{-1}(1-x_j)) \Big] \\
&~ \hspace{6.5cm} \big( \because \sinh(t) \leq Ct~\text{for $t$ bounded.} \big) \\
&~ \leq C h^2 \Big[ 1+ {(\ep + \mu b)}^{-2} y(x_j) \Big]
\end{align*}
Similarly, the estimations for $I_3$ and $I_4$ can be obtained, and upper bounds for these estimations have the same forms.

Now, for the second case when $h \geq \ep + \mu b $, the truncation error is given by (\ref{eq:5.5.4}). Then we have
\begin{align*}
\lvert I_1 \rvert~\leq &~~(\ep+\mu b)h^{-2}B(-P)\frac{1}{2}\left\lvert \int_{x_j}^{x_j- h}(x_j- h- t)^{2}z^{(3)}(t)\textrm{d}t \right\rvert \\
\leq &~~ \frac{(\ep + \mu b)}{2} \big(B(P) + P \big) \int_{x_j-h}^{x_j} \lvert z^{(3)}(t) \rvert \textrm{d}t \\
& \hspace{6cm} \big(\because B(-P) = P+B(P)\big) \\
\leq &~~ C\big( (\ep + \mu b) B(P)+b h \big) \int_{x_j-h}^{x_j} \big(1+ (\ep + \mu b)^{-2} y(t) \big) \textrm{d}t \\
\leq &~~ C\big( (\ep + \mu b)+b h \big) \int_{x_j-h}^{x_j} \big(1+ (\ep + \mu b)^{-2} y(t) \big) \textrm{d}t \\
& \hspace{6cm} \big(\because B(P) < 1~ \text{when}~ b> 0 \big) \\
\leq &~~ C\big( (\ep + \mu b)+ h \big) \int_{x_j-h}^{x_j} \big(1+ (\ep + \mu b)^{-2} \exp(-b{(\ep + \mu b)}^{-1}(1-t)) \big) 
\end{align*}
\begin{align*}
\leq &~~ C\big( (\ep + \mu b)+ h \big) \Big[ h+ b^{-1} (\ep + \mu b)^{-1} y(x_j)\big(1-\exp(-b h (\ep + \mu b)^{-1}) \big)\Big] \\
\leq &~~ C\big( (\ep + \mu b)+ h \big) \big[ h+ b^{-1} (\ep + \mu b)^{-1} y(x_j)\big] \\
& \hspace{6.2cm}\big(\because \exp(-b h (\ep + \mu b)^{-1}) > 0 \big) \\
\leq &~~ C\big( h+ h \big) \big[ h+ b^{-1} (\ep + \mu b)^{-1} y(x_j)\big] \\
\leq &~~ C\big[ h^2+ h (\ep + \mu b)^{-1} y(x_j) \big] \\
\leq &~~ C\Big[ h^2+ \big(\frac{\ep + \mu b}{h}\big) \big(\frac{h}{\ep + \mu b}\big)^2 y(x_j) \Big] \\
\leq &~~ C\Big[ h^2+ (\ep + \mu b) h^{-1} \exp\big(h (\ep + \mu b)^{-1}\big) \exp\big(-b{(\ep + \mu b)}^{-1}(1-x_j)\big) \Big] \\
& \hspace{6cm}\big(\because~ t^{k} \leq C e^{t}~\text{for +ive integer}~k \big) \\
\leq &~~ C\Big[ h^2+ (\ep + \mu b) h^{-1} \exp\big(b h (\ep + \mu b)^{-1}\big) \exp\big(-b{(\ep + \mu b)}^{-1}(1-x_j)\big) \Big] \\ 
\leq &~~ C\Big[ h^2+ (\ep + \mu b) h^{-1} y(x_{j+1}) \Big]
\end{align*}

\ni Similarly, remaining terms can be estimated, and upper bounds for the estimations have the same forms. These results for the truncation error can be written as following lemma.

\begin{lemma} \label{lemma:5.5}
Let $\phi(x)$ be the solution of (\ref{eq:5.3.1}) and let its first four derivatives exist, then for the truncation error, we have
\bse \label{eq:5.5.9}
\begin{align}
\lvert \uptau_j \rvert \leq&~ C h^2 + C h^2 (\ep + \mu b)^{-2} \exp(-b (\ep + \mu b)^{-1}(1-x_j)),~h \leq \ep + \mu b \label{eq:5.5.9a}\\
\lvert \uptau_j \rvert \leq&~ C h^2 + C (\ep + \mu b) h^{-1}  \exp(-b (\ep + \mu b)^{-1}(1-x_{j+1})),~h \geq \ep + \mu b \label{eq:5.5.9b}
\end{align}
\ese
\end{lemma}

\section{Convergence}
\label{sec:5.6}
The CFS is uniformly second-order convergent for the problem (\ref{eq:5.3.1}). This statement follows from the following theorem.
\begin{theorem} \label{thm:5.1}
There exists a constant $C$, independent of $ \ep + \mu b $ and $h$, such that
\be \label{eq:5.6.1}
\lvert e_j \rvert \leq~ C h^2, \hspace{1.6cm}
\ee
for all $(\ep + \mu b) \in (0,1]$ and $b >0$.
\end{theorem}
In order to prove this theorem, we have to do some preparations. We use the comparison approach \cite{kets,ajmb80,ajm81}. By this approach, we choose the comparison functions $\eta(x)= 1+x$ and $\xi(x)= \exp(-\lambda (\ep + \mu b)^{-1} (1-x))$ for some $\lambda >0$. We use the functions $\eta(x)$ and $\xi(x)$ to estimate the error, where $\phi(x)$ is well behaved and where it is not i.e., near the layer, respectively. Also, we use the lower bounds of $\mathscr{L}^{h}\eta(x_j)$ and $\mathscr{L}^{h}\xi(x_j)$ that is important here. For these lower bounds, we have two following lemmas.
\begin{lemma} \label{lemma:5.6}
There exists a constant $C$, independent of $ \ep + \mu b $ and $h$, such that
\begin{align} \label{eq:5.6.2}
\mathscr{L}^{h}\eta(x_j) \geq C, \hspace{2cm}
\end{align}
for all $ \ep + \mu b \in (0,1]$ and $b >0$.
\end{lemma}
\begin{proof}
Proof of this lemma is straightforward, so we omit this.
\end{proof}

\begin{lemma} \label{lemma:5.7}
There exist constants $c_1$ and $c_2$ such that $h \leq c_1$ and $0 < \lambda \leq c_2$, and for some constant $C$, it holds
\bse \label{eq:5.6.3}
\begin{align}
\mathscr{L}^{h}\xi(x_j)&~ \geq C (\ep + \mu b)^{-1}\xi(x_j),~h \leq \ep + \mu b, \label{eq:5.6.3a} \hspace{4cm}\\
\mathscr{L}^{h}\xi(x_j)&~ \geq C h^{-1}\xi(x_j),~h \geq \ep + \mu b. \label{eq:5.6.3b} \hspace{4cm}
\end{align}
\ese
\end{lemma}
\begin{proof}
Following the Lemma $3.6$ in \cite{ajm81}, we get the expression
\begin{align*}
\mathscr{L}^{h}\xi(x_j)=&~ -a_W \exp(-\lambda (\ep + \mu b)^{-1} (1-x_{j-1})) \\
& \hspace{2cm} + (a_W+a_E) \exp(-\lambda (\ep + \mu b)^{-1} (1-x_j))  \\
&~ \hspace{3.5cm} - a_E \exp(-\lambda (\ep + \mu b)^{-1} (1-x_{j+1})) \\
=&~ a_W \xi(x_{j+1})\left(\exp(-\lambda h (\ep + \mu b)^{-1})- a_E/a_W \right) \\
& \hspace{5cm} \Big(1- \exp(-\lambda h (\ep + \mu b)^{-1}) \Big).
\end{align*}
On estimating the individual factors in the above expression for the three cases $(i)~ h/(\ep + \mu b) \leq c$, $(ii)~ h/(\ep + \mu b) \geq C$ and $(iii)~ c \leq h/(\ep + \mu b) \leq C$ (for appropriately chosen $c$ and $C$), the required results follow.
\end{proof}
\begin{theorem} \label{thm:5.2}
Let $\{\phi_j\}$ be the approximate solution of (\ref{eq:5.3.1}) by CFS. Then there is constant $C$, independent of $(\ep + \mu b)$ and $h$, such that
\bse \label{eq:5.6.4}
\begin{align}
\lvert \phi(x_j)-\phi_j\rvert \leq &~ C h^2 + C h^2 (\ep + \mu b)^{-1} \exp(-\lambda (\ep + \mu b)^{-1} (1-x_j)),~h \leq \ep + \mu b, \label{eq:5.6.4a}\\
\lvert \phi(x_j)-\phi_j\rvert \leq &~ C h^2 + C (\ep + \mu b) \exp(-\lambda (\ep + \mu b)^{-1} (1-x_{j+1})),~h \geq \ep + \mu b. \label{eq:5.6.4b}
\end{align}
\ese
\end{theorem}
\begin{proof}
Case when $h \leq \ep + \mu b $. From (\ref{eq:5.5.9a}), by choosing a $\lambda \leq b$, we have
\begin{align*}
\lvert \uptau_j\rvert = &~ \lvert \mathscr{L}^{h}(\phi(x_j)-\phi_j)\rvert \\
\leq &~ C h^2 + C h^2 (\ep + \mu b)^{-2} \exp(-\lambda (\ep + \mu b)^{-1} (1-x_j)) \\
\leq &~ C h^2 \mathscr{L}^{h}\eta(x_j) + C h^2 (\ep + \mu b)^{-1} \mathscr{L}^{h}\xi(x_j)\\
=&~ \mathscr{L}^{h} \left[ C h^2 \eta(x_j) + C h^2 (\ep + \mu b)^{-1} \mathscr{L}^{h}\xi(x_j)\right].
\end{align*}
Then we see that (\ref{eq:5.6.4a}) follows from Lemma (\ref{lemma:5.1}). Similarly, the other case can be proved.
\end{proof}
\begin{lemma} \label{lemma:5.8}
The solution $\phi(x)$ of (\ref{eq:5.3.1}) can be written in the form
\be \label{eq:5.6.5}
\phi(x) = A_0(x) + B_0 \exp(-b (\ep + \mu b)^{-1} (1-x)) + (\ep + \mu b) R_0(x; \ep + \mu b), \hspace{1cm}
\ee
where the constant $B_0$ and the norm $A_0 \in C^{m+1}[0,1]$ depend on the boundary values of (\ref{eq:5.3.1}) and the integral of
$s \in C^m[0,1]$. The function $R_0(x)$ satisfies the following problem:
\begin{align} \label{eq:5.6.6}
-(\ep + \mu b) R^{\prime\prime}_0 + b R^{\prime}_0 = F_0(x),~R_0(0; \ep + \mu b) = \kappa_0(\ep + \mu b),~R_0(1;\ep + \mu b) = 0,
\end{align}
where $\kappa_0(\ep + \mu b)$ is bounded and $F_0 \in C^{m-1}[0,1]$.
\end{lemma}
\begin{proof}
We omit the proof here. For the same, see \cite{smt75}.
\end{proof}

Now, after all this preparation, we can prove Theorem \ref{thm:5.1}.
\begin{proof}
From Lemma \ref{lemma:5.8}, we can say that the solution $\phi(x)$ of (\ref{eq:5.3.1}) has a three term decomposition (\ref{eq:5.6.5}) also. Therefore, Theorem \ref{thm:5.1} holds if the contribution of each term of this decomposition to the discretization error is uniformly $O(h^2)$. The derivatives of the first term $A_0$ are uniformly bounded, therefore, by Lemma \ref{lemma:5.3} and Lemma \ref{lemma:5.2}, it follows that the contribution of $A_0$ is bounded by $C h^2$. For the third term $(\ep + \mu b) R_0$, from Lemma \ref{lemma:5.8} along with Theorem \ref{thm:5.2}, it follows that the contribution of $(\ep + \mu b) R_0$ is uniformly $O(h^2)$ as the discretization error of $R_0$ from using the CFS has the estimation (\ref{eq:5.6.4a}) and (\ref{eq:5.6.4b}). The second term of (\ref{eq:5.6.5}) is the analytical solution of the homogeneous equation of (\ref{eq:5.3.1}), so its contribution to the discretization error is zero. This completes the proof.
\end{proof}

\section{Numerical Results}
\label{sec:5.7}
In this section, we implement the complete flux scheme (CFS) to some example problems. We give some plots for the solutions and the $\ep$-effect on the solutions. Also, we give plots for the error to show the convergence.

\begin{example} \label{exp:5.9}
Consider the following elliptic SPDDE with appropriate B.C.
\bse \label{eq:5.7.9}
\begin{align}
-\ep &\phi^{\prime\prime}(x) + b\phi^{\prime}(x-\mu)= s, ~ \fa x \in\Omega, \hspace{2cm} \label{eq:5.7.9a} \\
&\phi(0) = \phi_L,~~\phi(1)= \phi_R,  \label{eq:5.7.9b}
\end{align}
\ese
where $0 < \ep \ll 1$, $\mu =0.1 \ep$, $b = 1$, $s =0$, $\phi_L = 1$, $\phi_R = 0$ and $\Omega = (0,1)$.
The exact solution of the corresponding approximate SPP is given by
\begin{align*}
\phi(x) = \frac{1-e^{-(1-x)/(\ep+\mu)}}{1 -e^{-1/(\ep+\mu)}}.
\end{align*}
\end{example}

\begin{figure}
\subfloat[]{\includegraphics[width=6cm,height=6cm]{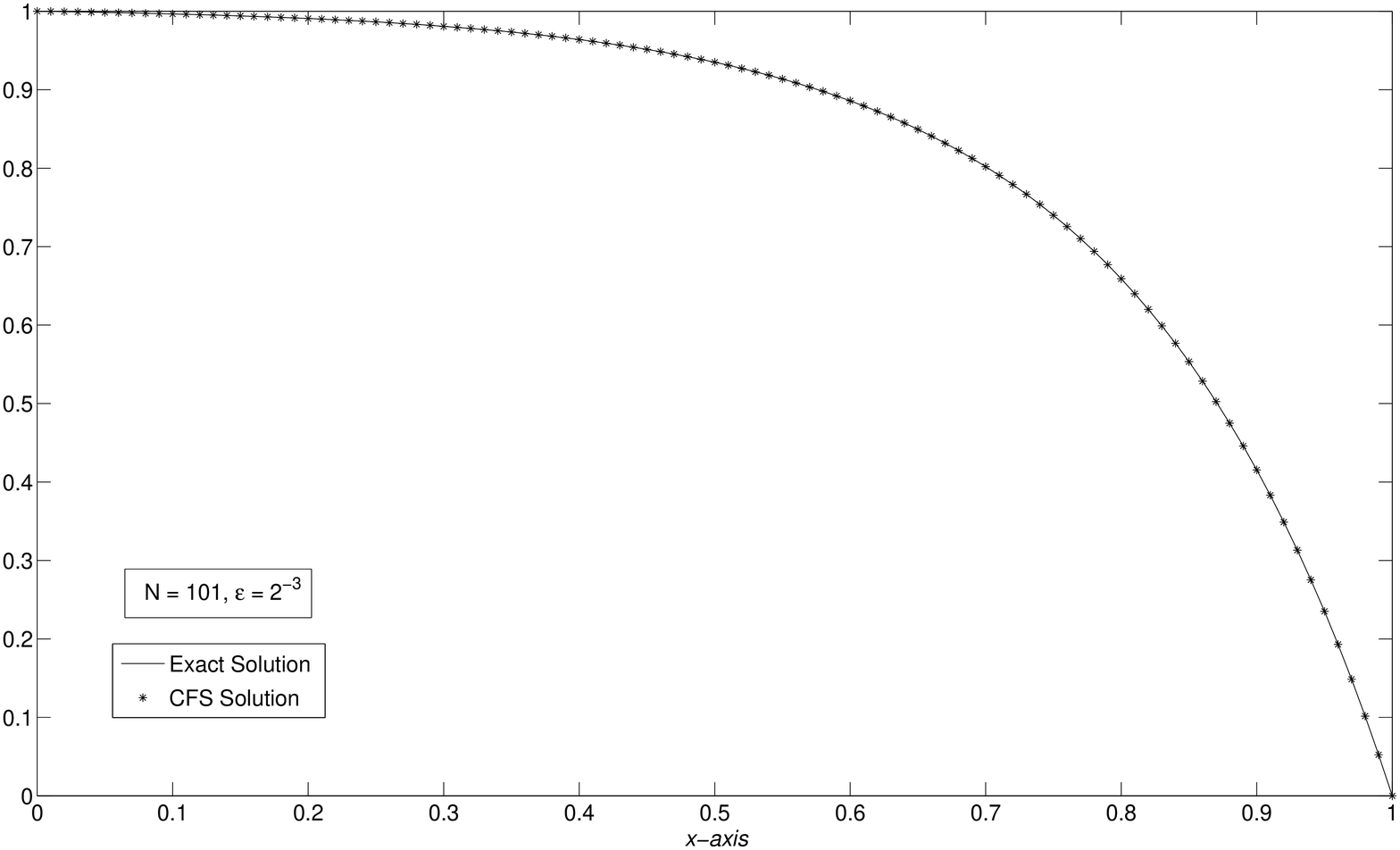}\label{1a}}
\subfloat[]{\includegraphics[width=6cm,height=6cm]{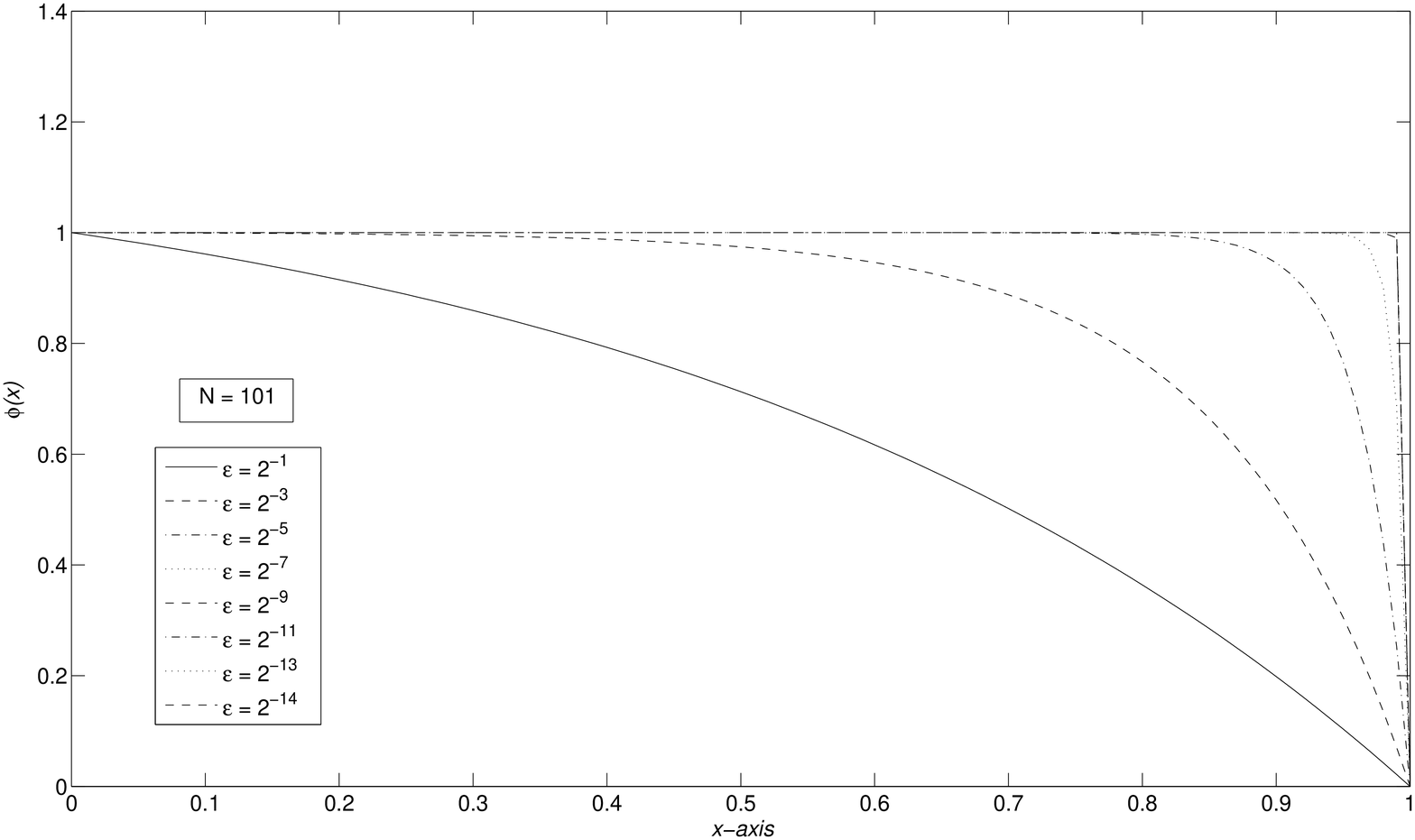}\label{1b}}\\
\subfloat[]{\includegraphics[width=6cm,height=6cm]{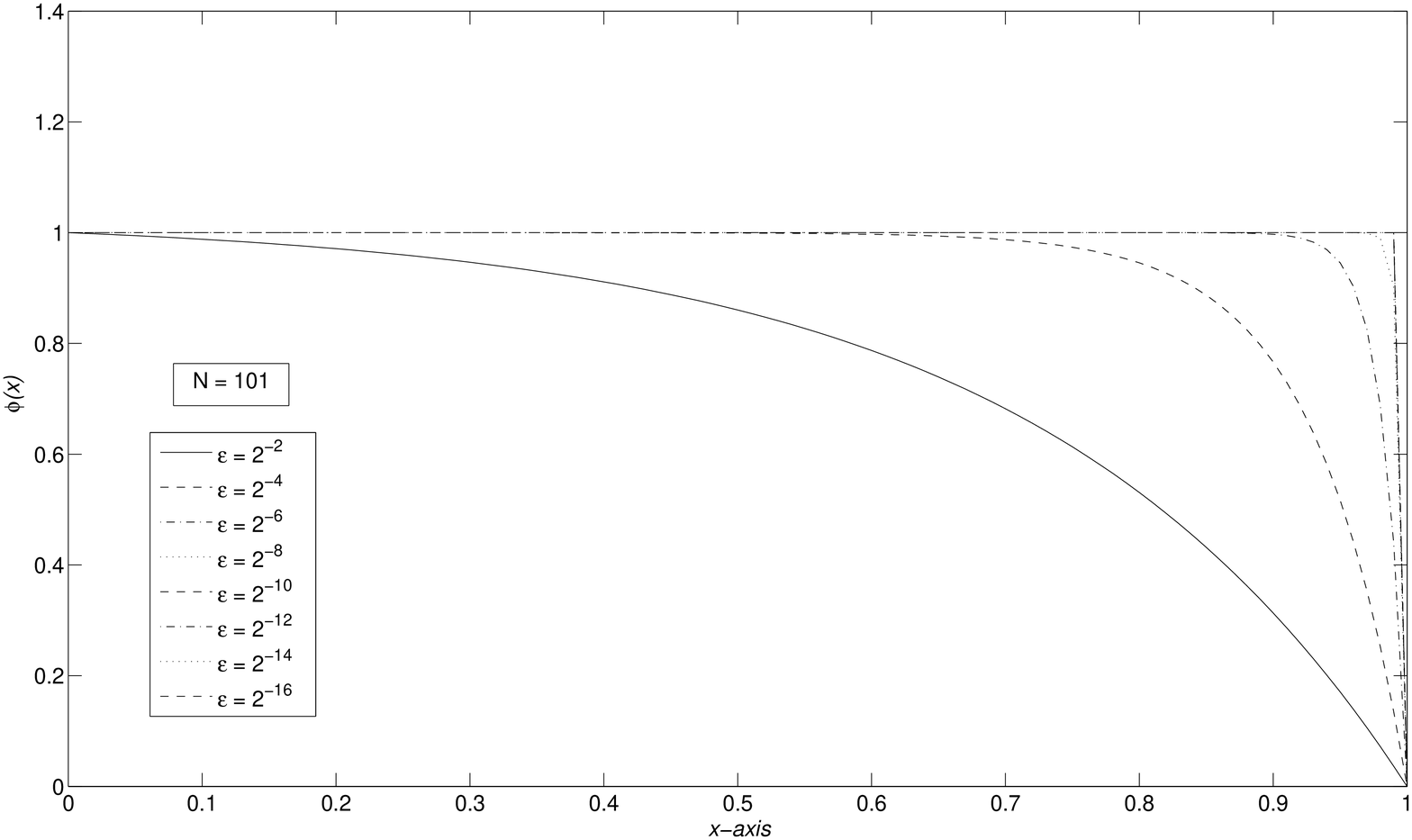}\label{1c}}
\subfloat[]{\includegraphics[width=6cm,height=6cm]{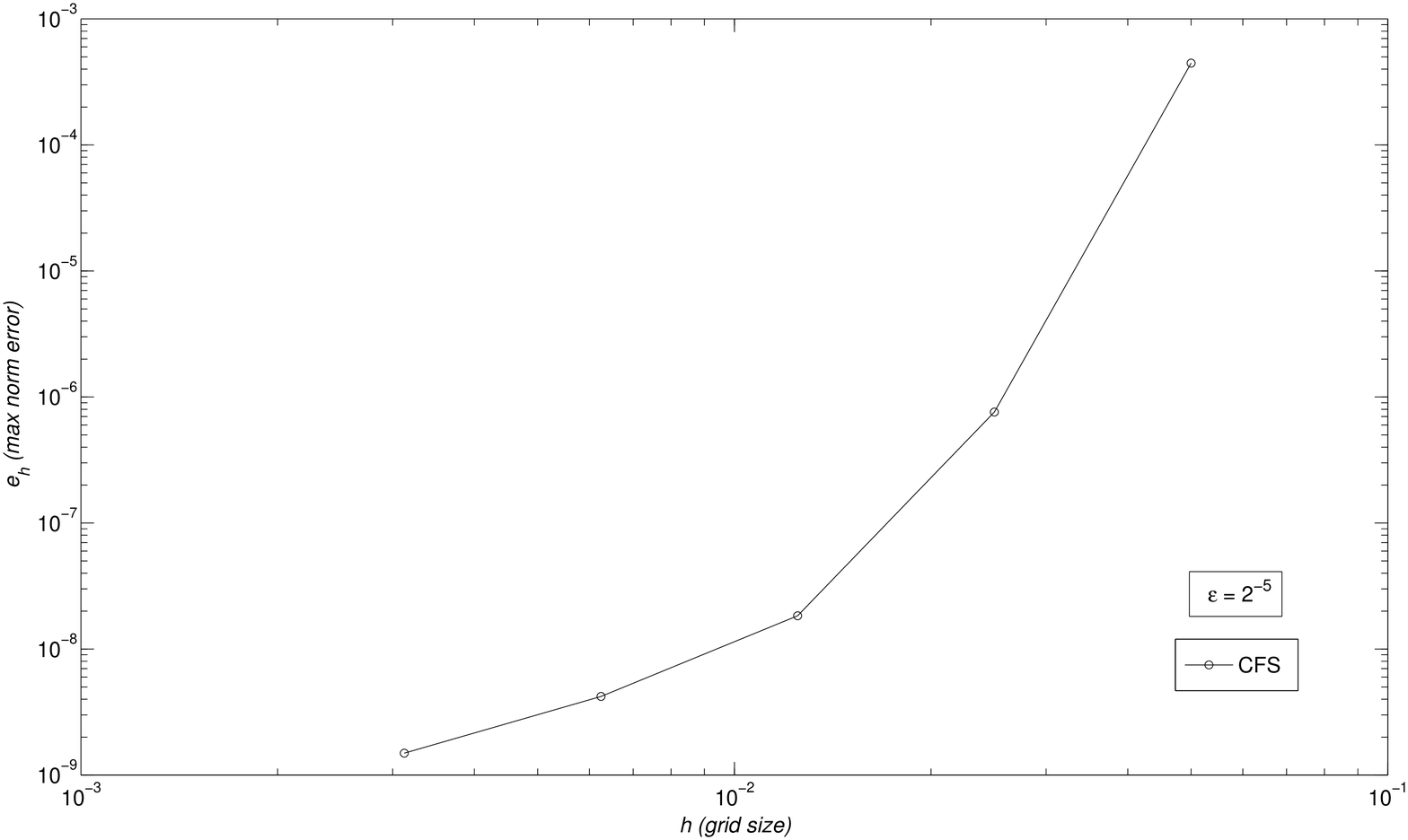}\label{1d}}
\caption{(a) Comparison of the exact $\&$ CFS solutions, (b-c) $\ep$-effects and (d) Log-log graph of the discretization error $e_h$ as a function of the grid size $h$.}
\label{fig:5.9}
\end{figure}

Here, the solution has a thin boundary layer of width $\ep$ near the boundary $x=1$. The comparison of the exact and numerical solutions is as shown in the figure \ref{fig:5.9}(a). Let $h = \Delta x = 1/(N-1)$ be the grid (step) size with $N$ number of grid points. The number of grid points $N=101$ are adequate to capture the boundary layer accurately, in this example. The effects of different values of singular perturbation parameter $\ep$ are shown in the figures \ref{fig:5.9}(b)-\ref{fig:5.9}(c) and from these figures, one can easily notice that as the singular perturbation parameter $\ep$ goes smaller and smaller, the boundary layers become sharper and sharper. The discretization error $e_h$ is calculated in the max norm for different grid sizes $h = 0.05, 0.025, 0.0125, 0.00625, 0.003125$. Log-log graph for the discretization error $e_h$, showing convergence, is shown in the figure \ref{fig:5.9}(d).

\begin{figure}
\subfloat[]{\includegraphics[width=6cm,height=6cm]{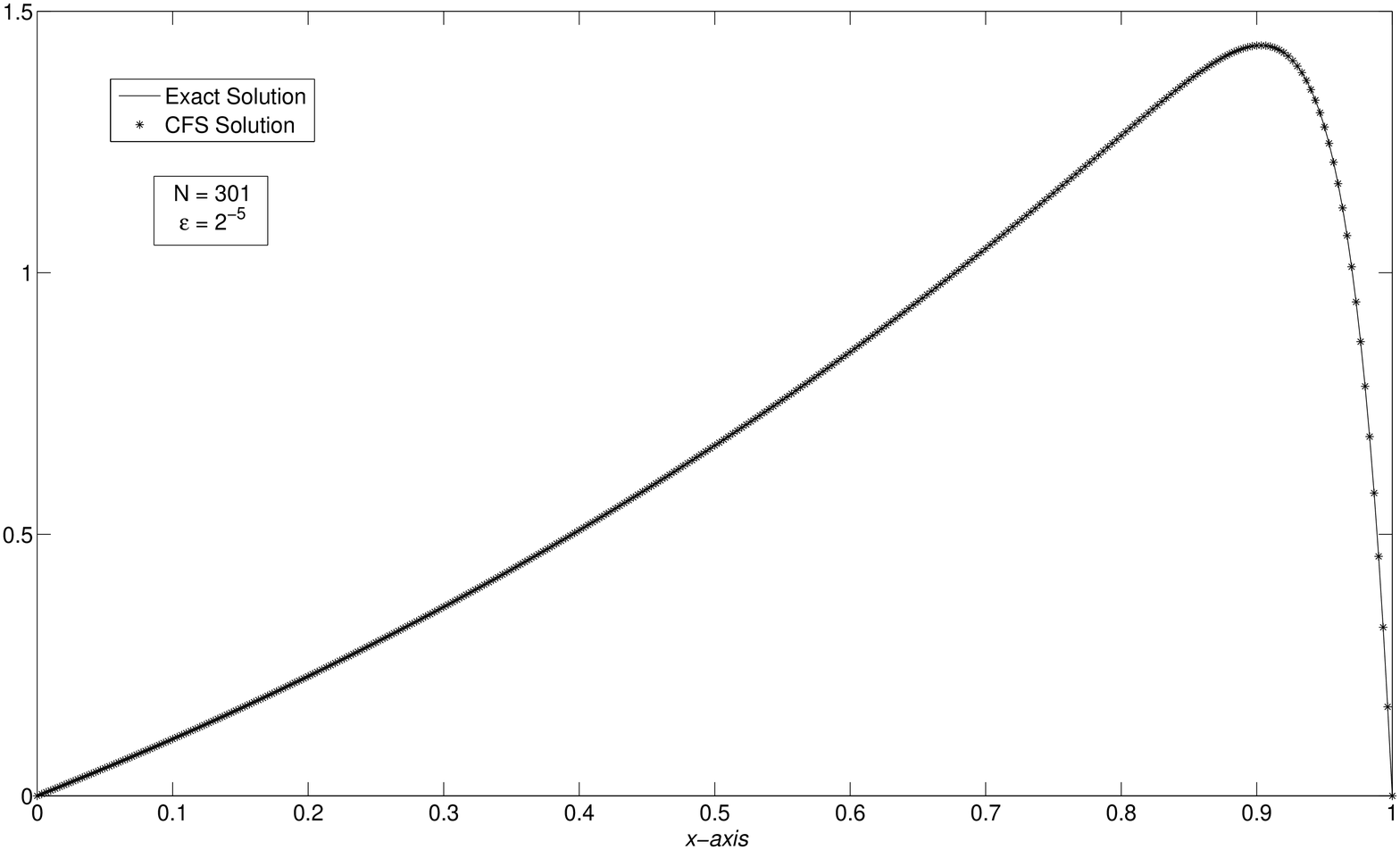}\label{2a}}
\subfloat[]{\includegraphics[width=6cm,height=6cm]{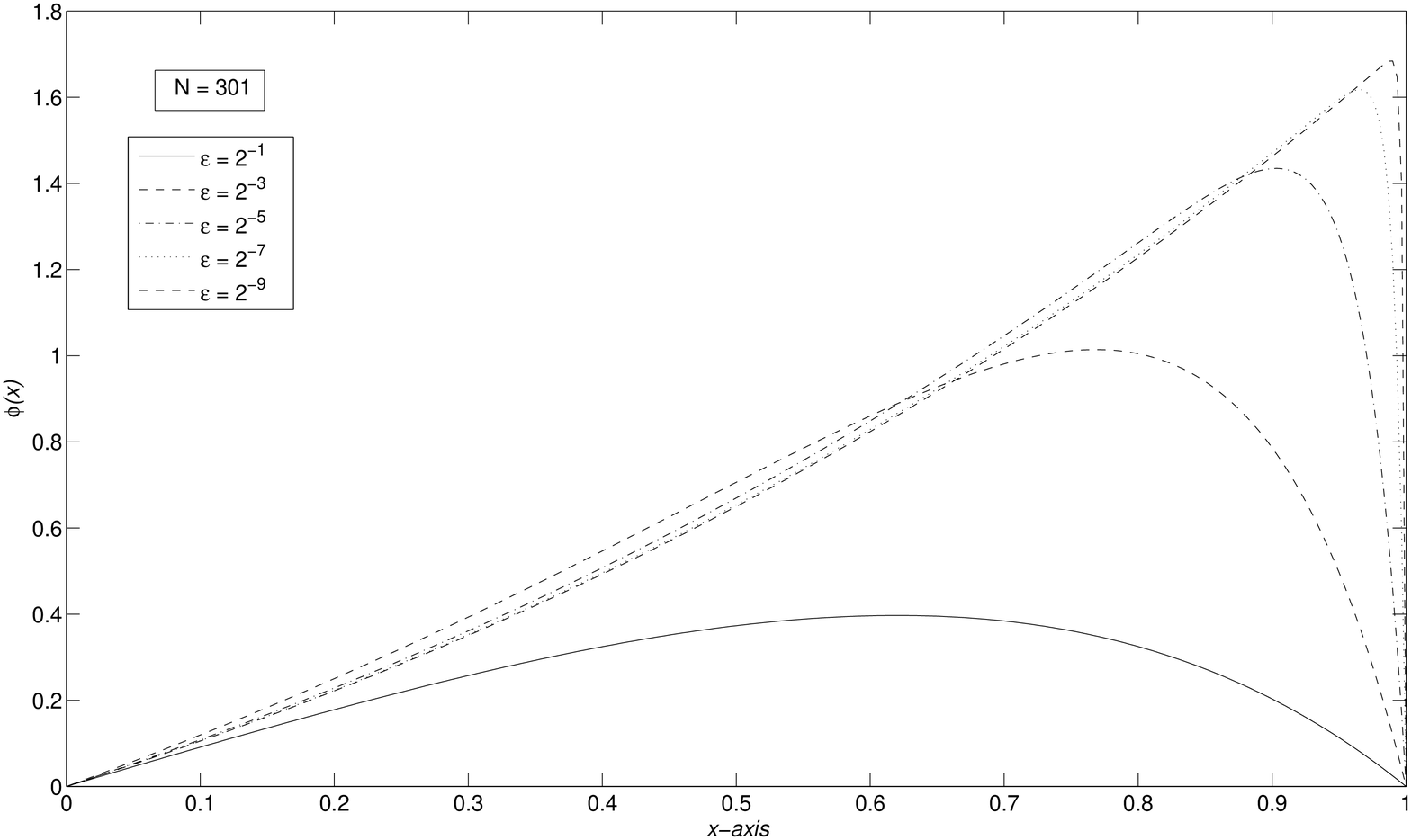}\label{2b}}\\
\subfloat[]{\includegraphics[width=6cm,height=6cm]{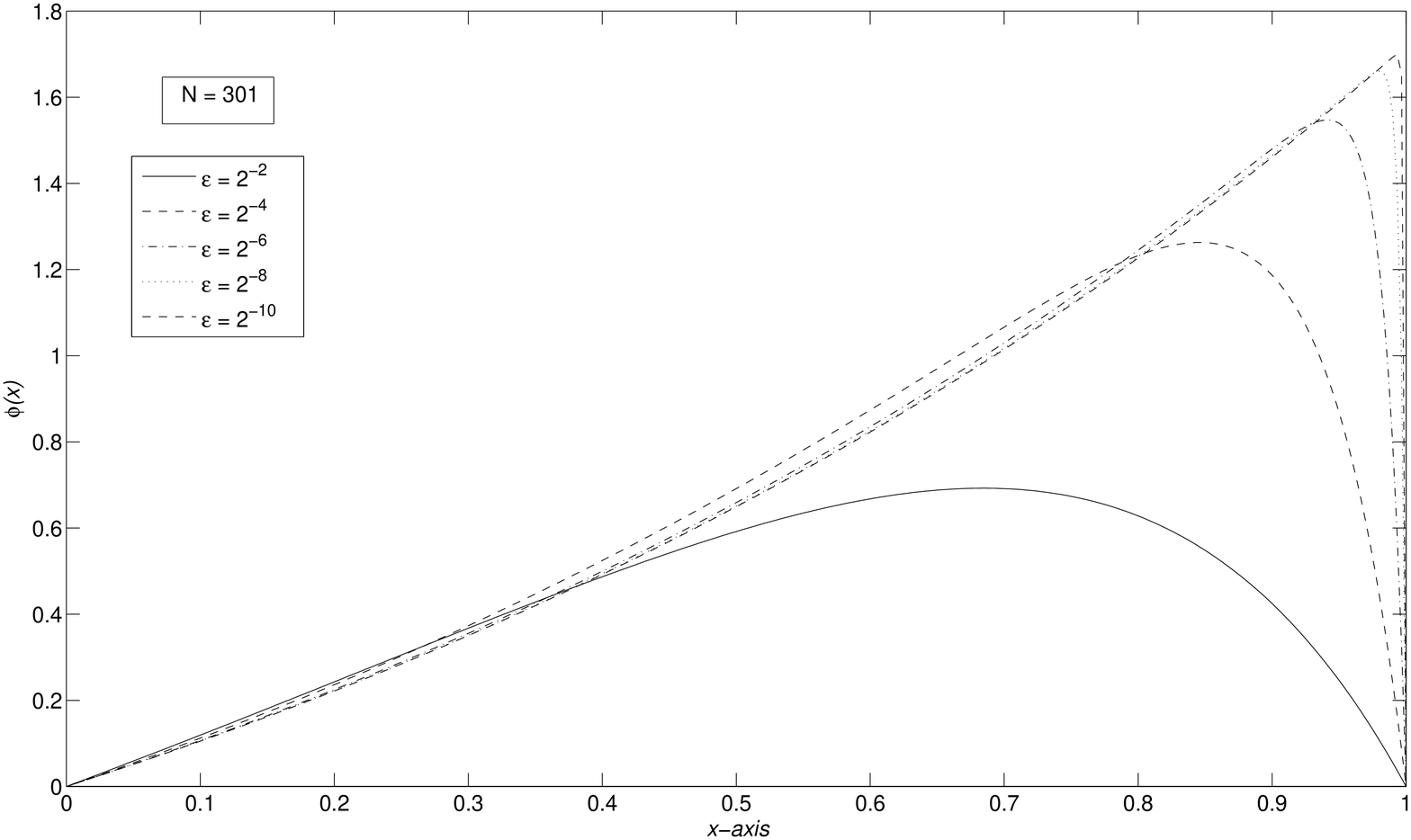}\label{2c}}
\subfloat[]{\includegraphics[width=6cm,height=6cm]{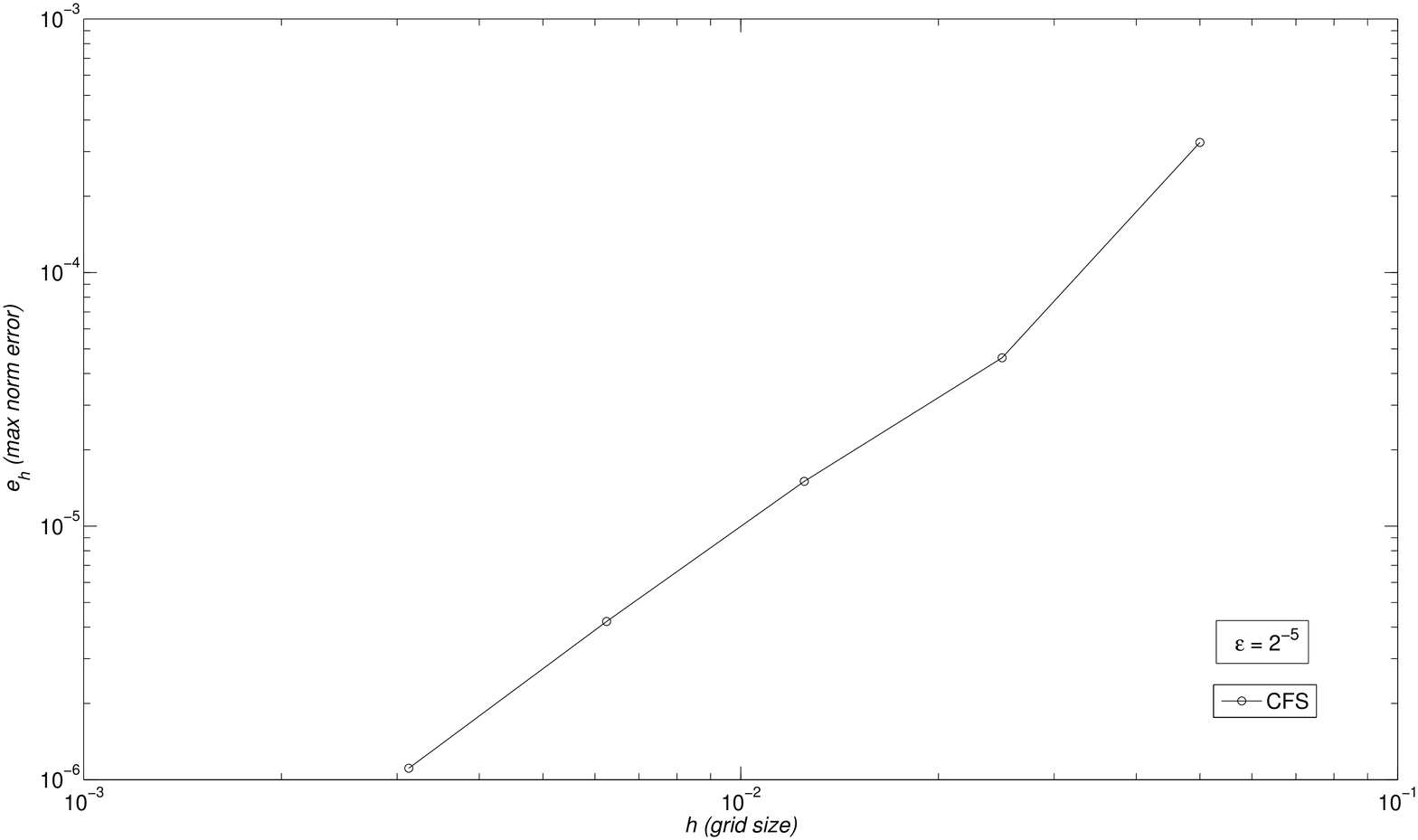}\label{2d}}
\caption{(a) Comparison of the exact $\&$ CFS solutions, (b-c) $\ep$-effects and (d) Log-log graph of the discretization error $e_h$ as a function of the grid size $h$.}
\label{fig:5.3}
\end{figure}

\begin{example} \label{exp:5.3}
Consider the following elliptic SPDDE with appropriate B.C.
\bse \label{eq:5.7.3}
\begin{align}
-\ep &\phi^{\prime\prime}(x) + b\phi^{\prime}(x-\mu)= s, ~ \fa x \in\Omega, \hspace{2cm} \label{eq:5.7.3a} \\
&\phi(0) = ~ \phi_L,~~\phi(1)= \phi_R,  \label{eq:5.7.3b}
\end{align}
\ese
where $0 < \ep \ll 1$, $\mu =0$, $b = 1$, $s =e^{x}$, $\phi_L = 0$, $\phi_R = 0$ and $\Omega = (0,1)$.
The exact solution of the corresponding approximate SPP is given by
\begin{align*}
\phi(x) = \frac{1}{1-\ep}\left( e^{x}-\frac{1-e^{1-{1/\ep}}-(1-e^{1})e^{(x-1)/\ep}}{1-e^{-1/\ep}} \right).
\end{align*}
\end{example}

Here, the solution has a thin boundary layer of width $\ep$ near the boundary $x=1$. The comparison of the exact and numerical solutions is as shown in the figure \ref{fig:5.3}(a). Let $h = \Delta x = 1/(N-1)$ be the grid (step) size with $N$ number of grid points. The number of grid points $N=301$ are adequate to capture the boundary layer accurately, in this example. The effects of different values of singular perturbation parameter $\ep$ are shown in the figures \ref{fig:5.3}(b)-\ref{fig:5.3}(c) and from these figures, one can easily notice that as the singular perturbation parameter $\ep$ goes smaller and smaller, the boundary layers become sharper and sharper. The discretization error $e_h$ is calculated in the max norm for different grid sizes $h = 0.05, 0.025, 0.0125, 0.00625, 0.003125$. Log-log graph for the discretization error $e_h$, showing convergence, is shown in the figure \ref{fig:5.3}(d).

\begin{example} \label{exp:5.4}
Consider the following elliptic SPDDE with appropriate B.C.
\bse \label{eq:5.7.4}
\begin{align}
-\ep &\phi^{\prime\prime}(x) + b\phi^{\prime}(x-\mu)= s, ~ \fa x \in\Omega, \hspace{2cm} \label{eq:5.7.4a} \\
&\phi(0) = ~ \phi_L,~~\phi(1)= \phi_R,  \label{eq:5.7.4b}
\end{align}
\ese
where $0 < \ep \ll 1$, $\mu =0.2 \ep$, $b = -1$, $s =-\phi(x)$, $\phi_L = 1$, $\phi_R = 1$ and $\Omega = (0,1)$.
The exact solution of the corresponding approximate SPP is given by
\begin{align*}
\phi(x) = \frac{(1-e^{m_2})e^{m_1 x}+(e^{m_1}-1)e^{m_2 x}}{e^{m_1 } - e^{m_2 }},
\end{align*}
where $m_1 = \frac{-1+\sqrt{1+4(\ep-\mu)}}{2(\ep-\mu)}$ and $m_2 = \frac{-1-\sqrt{1+4(\ep-\mu)}}{2(\ep-\mu)}$.
\end{example}

Here, the solution has a thin boundary layer of width $\ep$ near the boundary $x=0$. The comparison of the exact and numerical solutions is as shown in the figure \ref{fig:5.4}(a). Let $h = \Delta x = 1/(N-1)$ be the grid (step) size with $N$ number of grid points. The number of grid points $N=101$ are adequate to capture the boundary layer accurately, in this example. The effects of different values of singular perturbation parameter $\ep$ are shown in the figures \ref{fig:5.4}(b)-\ref{fig:5.4}(c) and from these figures, one can easily notice that as the singular perturbation parameter $\ep$ goes smaller and smaller, the boundary layers become sharper and sharper. The discretization error $e_h$ is calculated in the max norm for different grid sizes $h = 0.05, 0.025, 0.0125, 0.00625, 0.003125$. Log-log graph for the discretization error $e_h$, showing convergence, is shown in the figure \ref{fig:5.4}(d).

\begin{figure}
\subfloat[]{\includegraphics[width=6cm,height=6cm]{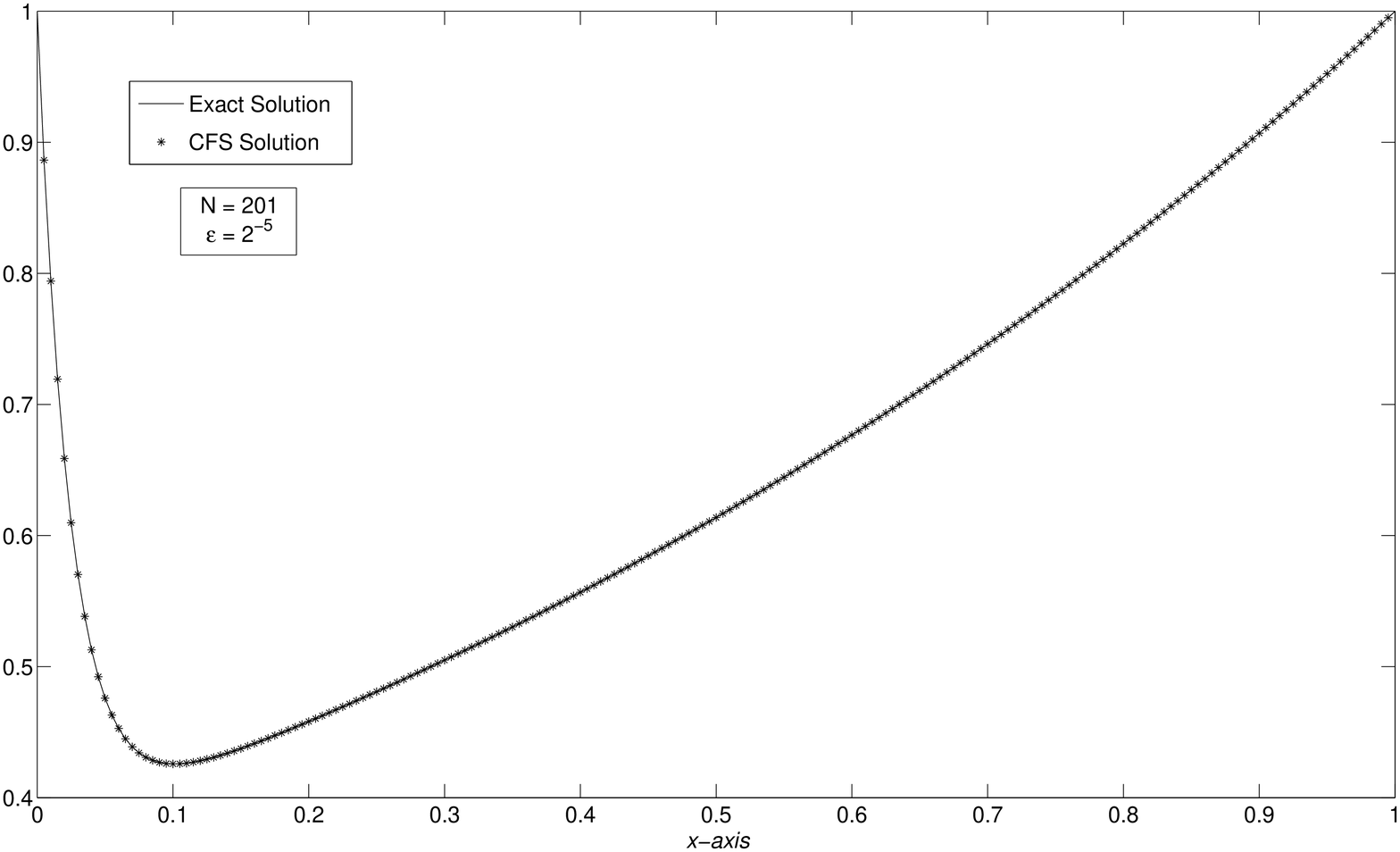}\label{3a}}
\subfloat[]{\includegraphics[width=6cm,height=6cm]{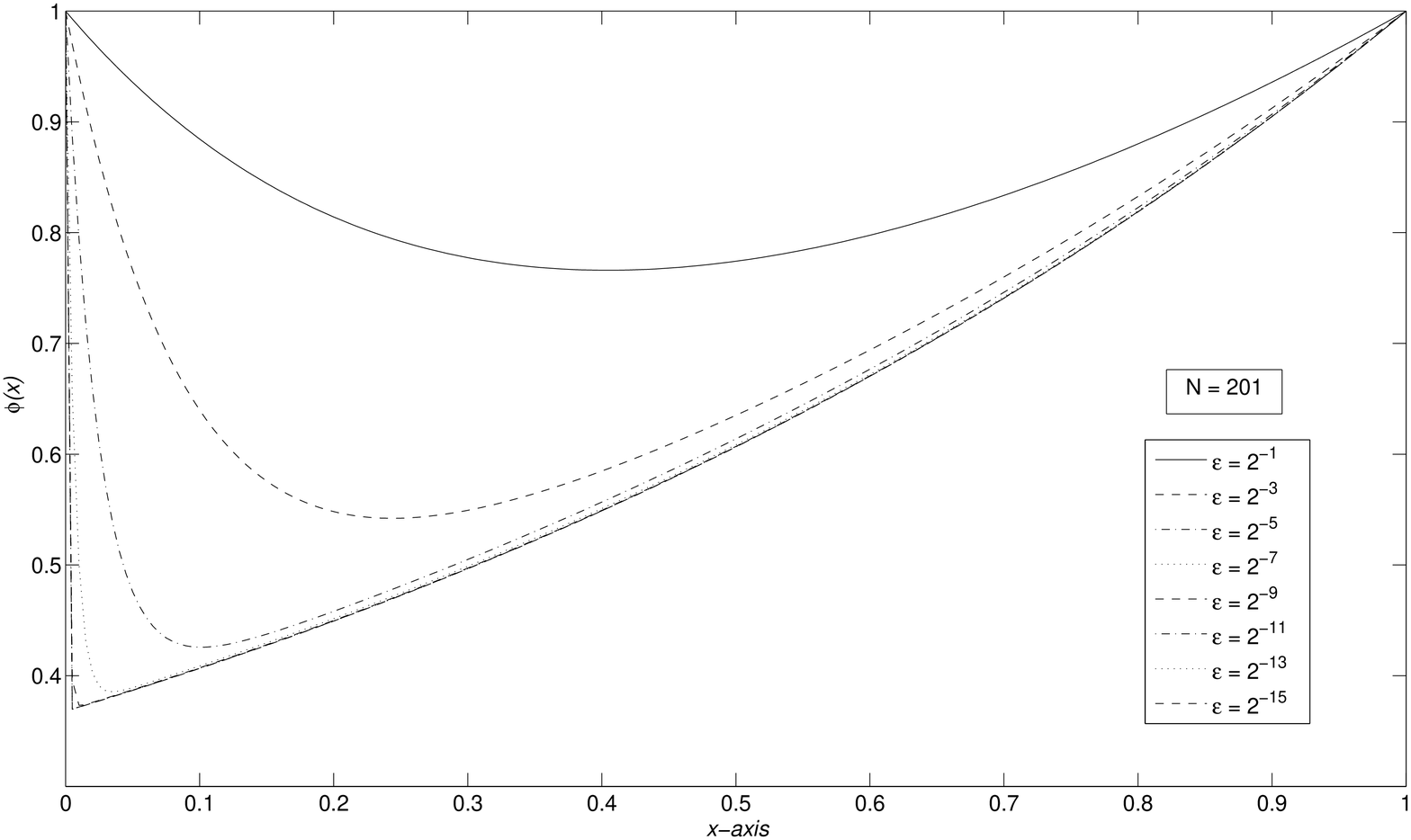}\label{3b}}\\
\subfloat[]{\includegraphics[width=6cm,height=6cm]{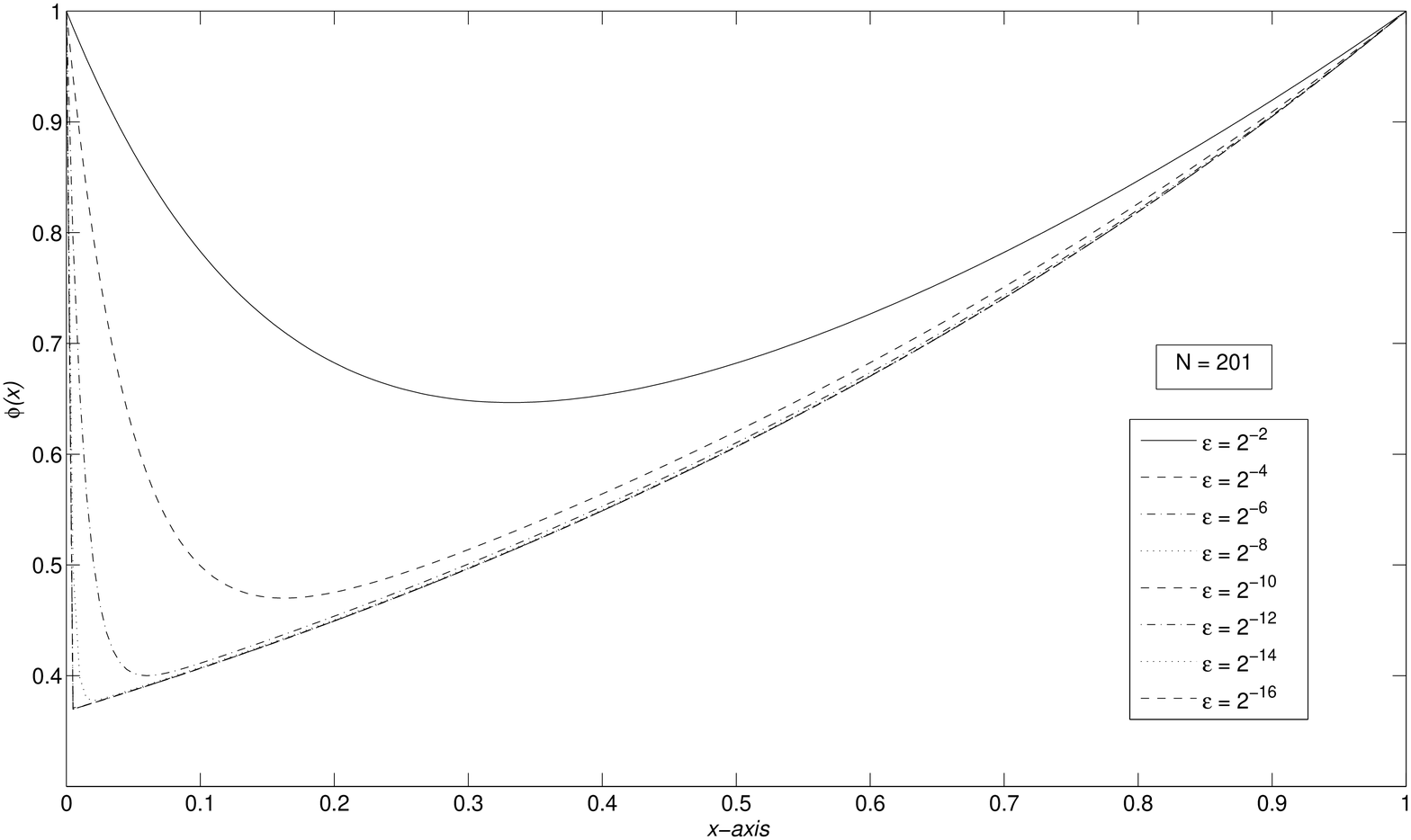}\label{3c}}
\subfloat[]{\includegraphics[width=6cm,height=6cm]{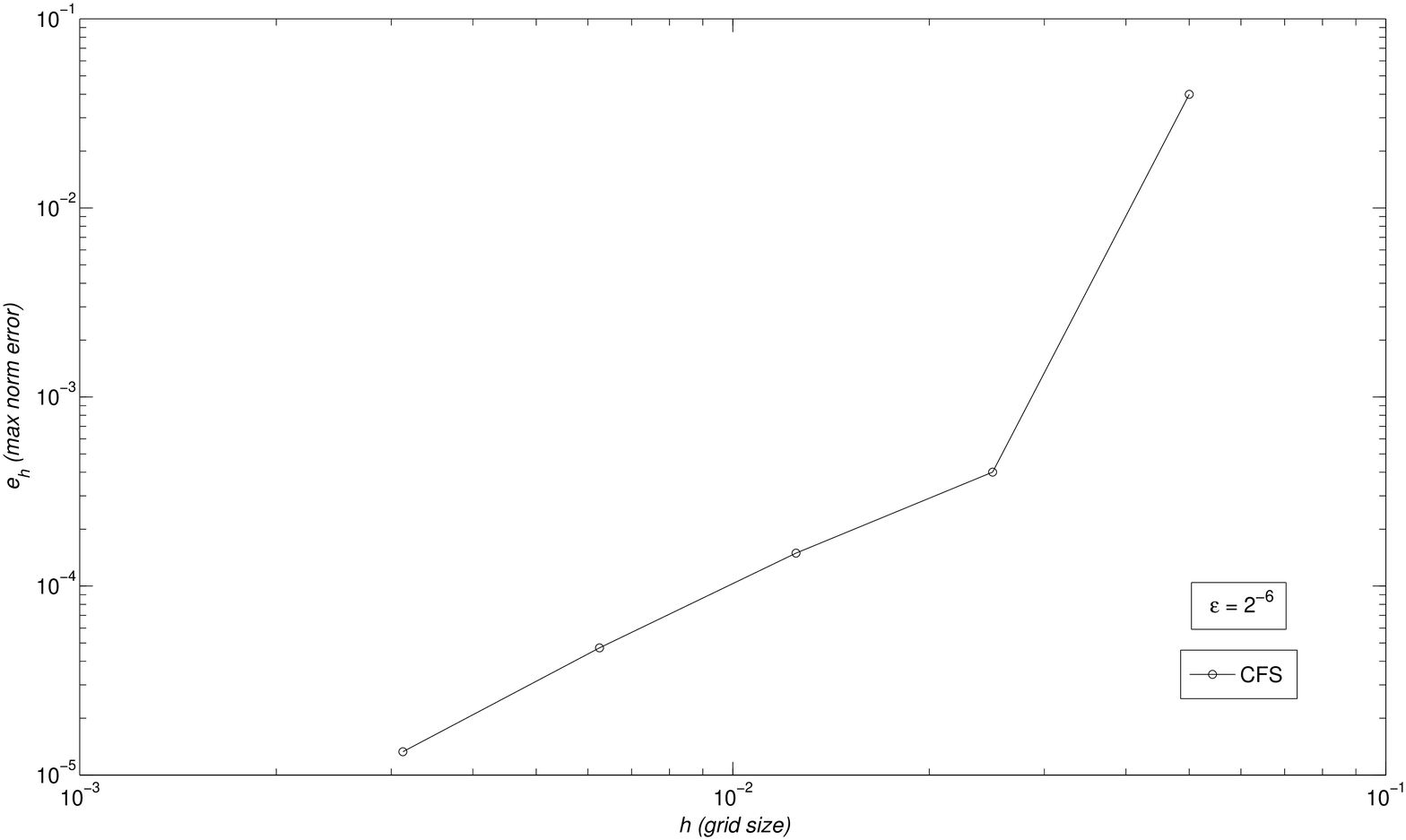}\label{3d}}
\caption{(a) Comparison of the exact $\&$ CFS solutions, (b-c) $\ep$-effects and (d) Log-log graph of the discretization error $e_h$ as a function of the grid size $h$.}
\label{fig:5.4}
\end{figure}

\begin{example} \label{exp:5.2}
Consider the following elliptic SPDDE with appropriate B.C.
\bse \label{eq:5.7.2}
\begin{align}
-\ep &\phi^{\prime\prime}(x) + b\phi^{\prime}(x-\mu)= s, ~ \fa x \in\Omega, \hspace{2cm} \label{eq:5.7.2a} \\
&\phi(0) = \phi_L,~~\phi(1)= \phi_R,  \label{eq:5.7.2b}
\end{align}
\ese
where $0 < \ep \ll 1$, $\mu =0$, $b = 1$, $s =-(1+\ep)\phi$, $\phi_L = 1 + e^{-(1 + \ep)/\ep}$, $\phi_R = 1 + e^{-1}$ and $\Omega = (0,1)$.
The exact solution of the corresponding approximate SPP is given by
\begin{align*}
\phi(x) = e^{(1+\ep)(x-1)/\ep} + e^{-x}.
\end{align*}
\end{example}

Here, the solution has a thin boundary layer of width $\ep$ near the boundary $x=1$ as. The comparison of the exact and numerical solutions is shown in the figure \ref{fig:5.2}(a). Let $h = \Delta x = 1/(N-1)$ be the grid (step) size with $N$ number of grid points. The number of grid points $N=101$ are adequate to capture the boundary layer accurately, in this example. The effects of different values of singular perturbation parameter $\ep$ are shown in the figures \ref{fig:5.2}(b)-\ref{fig:5.2}(c) and from these figures, one can easily notice that as the singular perturbation parameter $\ep$ goes smaller and smaller, the boundary layers become sharper and sharper. The discretization error $e_h$ is calculated in the max norm for different grid sizes $h = 0.05, 0.025, 0.0125, 0.00625, 0.003125$. Log-log graph for the discretization error $e_h$, showing convergence, is shown in the figure \ref{fig:5.2}(d).

\begin{figure}
\subfloat[]{\includegraphics[width=6cm,height=6cm]{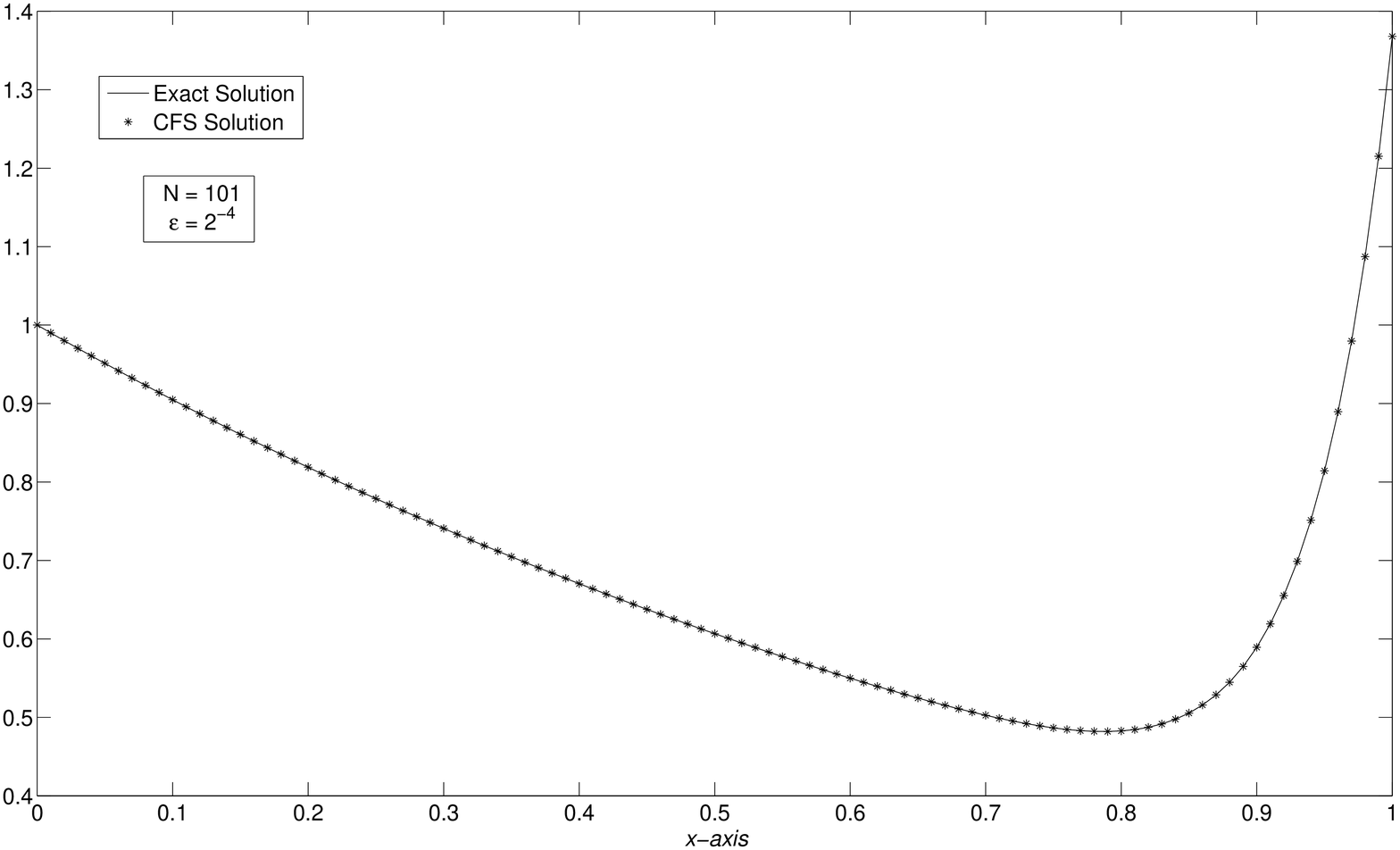}\label{4a}}
\subfloat[]{\includegraphics[width=6cm,height=6cm]{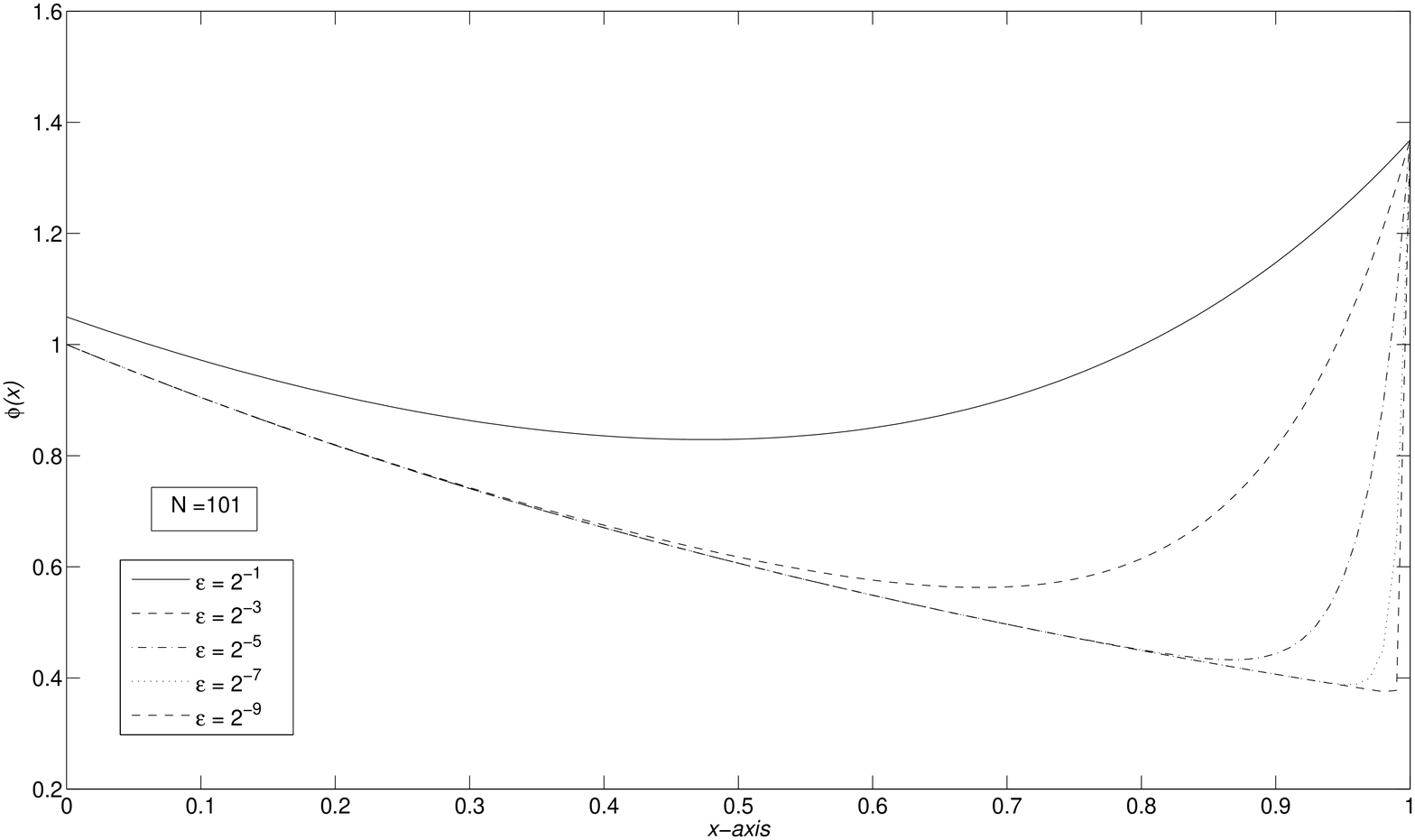}\label{4b}}\\
\subfloat[]{\includegraphics[width=6cm,height=6cm]{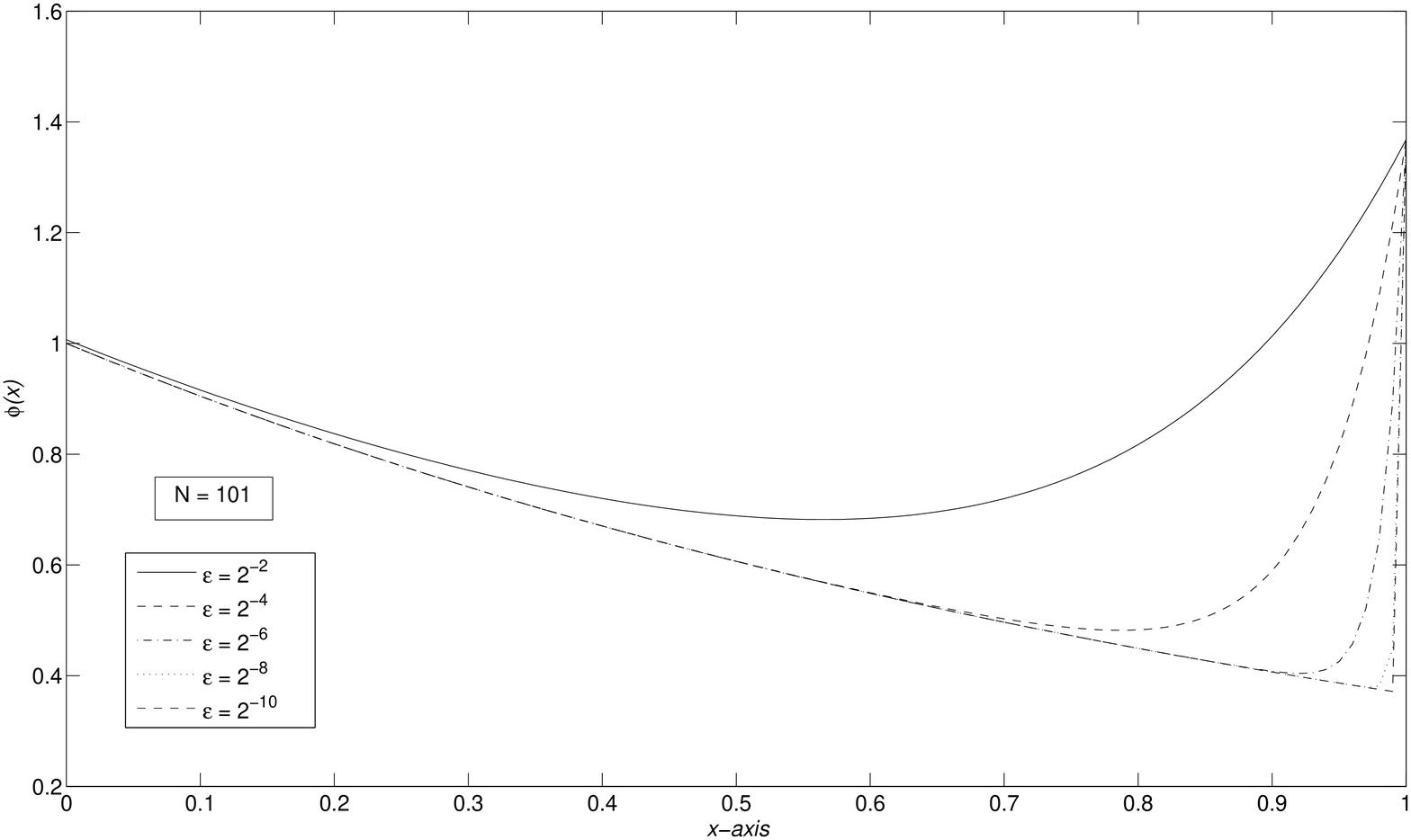}\label{4c}}
\subfloat[]{\includegraphics[width=6cm,height=6cm]{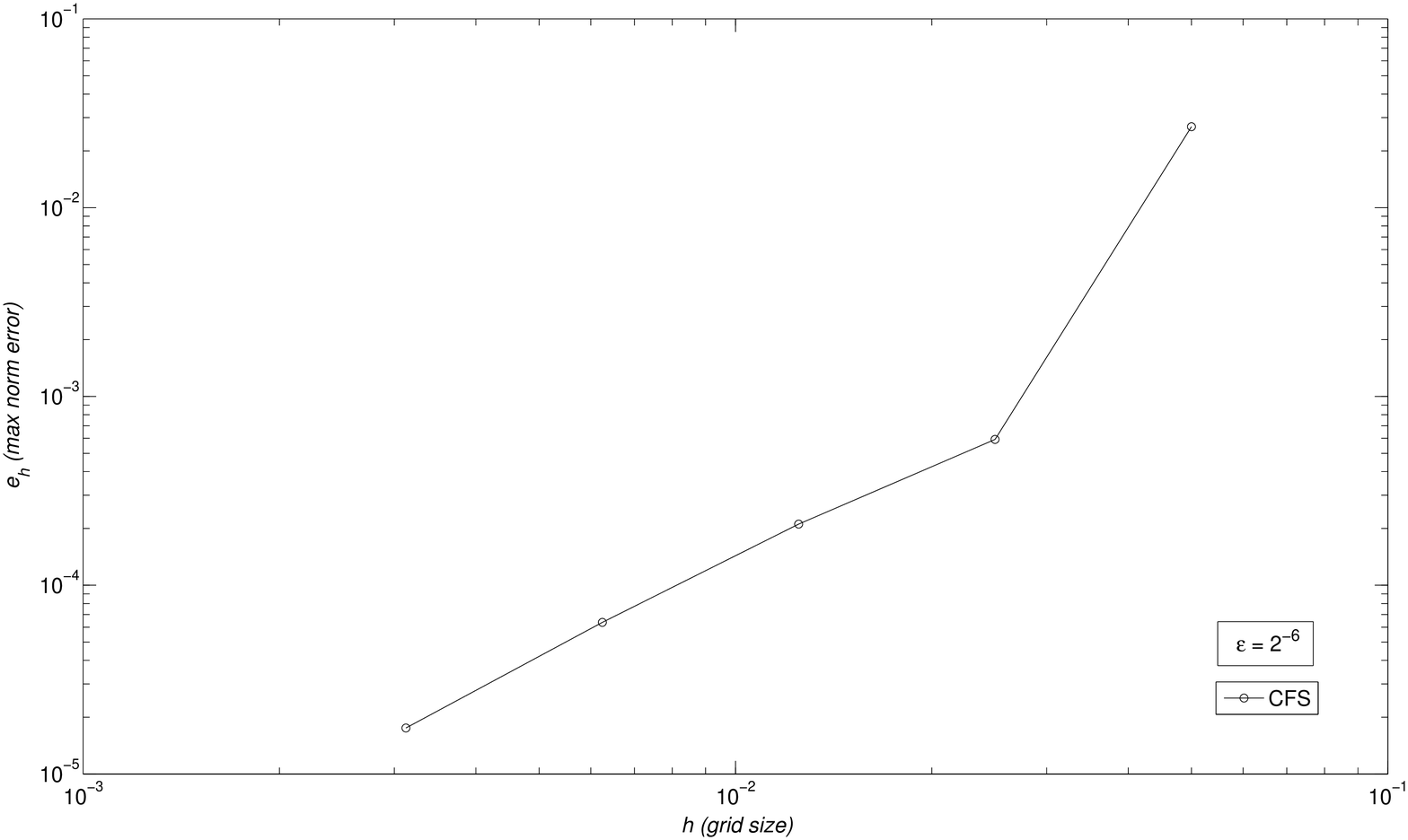}\label{4d}}
\caption{(a) Comparison of the exact $\&$ CFS solutions, (b-c) $\ep$-effects and (d) Log-log graph of the discretization error $e_h$ as a function of the grid size $h$.}
\label{fig:5.2}
\end{figure}

\begin{example} \label{exp:5.5}
Consider the following elliptic SPDDE with appropriate B.C.
\bse \label{eq:5.7.5}
\begin{align}
-\ep &\phi^{\prime\prime}(x) + b\phi^{\prime}(x-\mu)= s, ~ \fa x \in\Omega, \hspace{2cm} \label{eq:5.7.5a} \\
&\phi(0) = ~ \phi_L,~~\phi(1)= \phi_R,  \label{eq:5.7.5b}
\end{align}
\ese
where $0 < \ep \ll 1$, $\mu =0$, $b = -1$, $s =-(1+2 x)$, $\phi_L = 0$, $\phi_R = 1$ and $\Omega = (0,1)$.
The exact solution of the corresponding approximate SPP is given by
\begin{align*}
\phi(x) = x(x+1-2\ep) + (2\ep -1)\frac{1-e^{-x/\ep}}{1-e^{-1/\ep}}.
\end{align*}
\end{example}

Here, the solution has a thin boundary layer of width $\ep$ near the boundary $x=1$. The comparison of the exact and numerical solutions is as shown in the figure \ref{fig:5.5}(a). Let $h = \Delta x = 1/(N-1)$ be the grid (step) size with $N$ number of grid points. The number of grid points $N=201$ are adequate to capture the boundary layer accurately, in this example. The effects of different values of singular perturbation parameter $\ep$ are shown in the figures \ref{fig:5.5}(b)-\ref{fig:5.5}(c) and from these figures, one can easily notice that as the singular perturbation parameter $\ep$ goes smaller and smaller, the boundary layers become sharper and sharper. The discretization error $e_h$ is calculated in the max norm for different grid sizes $h = 0.05, 0.025, 0.0125, 0.00625, 0.003125$. Log-log graph for the discretization error $e_h$, showing convergence, is shown in the figure \ref{fig:5.5}(d).

\begin{figure}
\subfloat[]{\includegraphics[width=6cm,height=6cm]{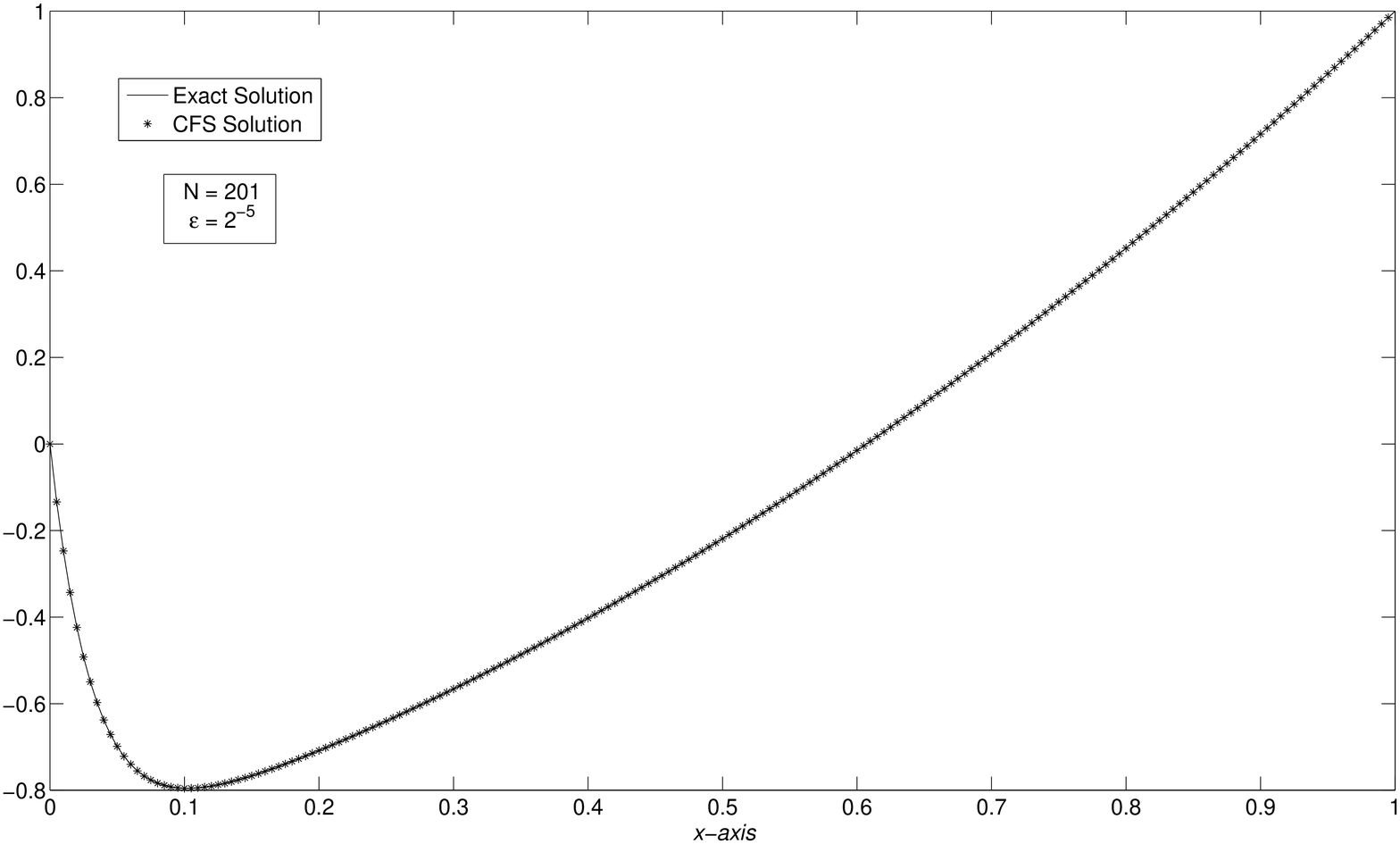}\label{5a}}
\subfloat[]{\includegraphics[width=6cm,height=6cm]{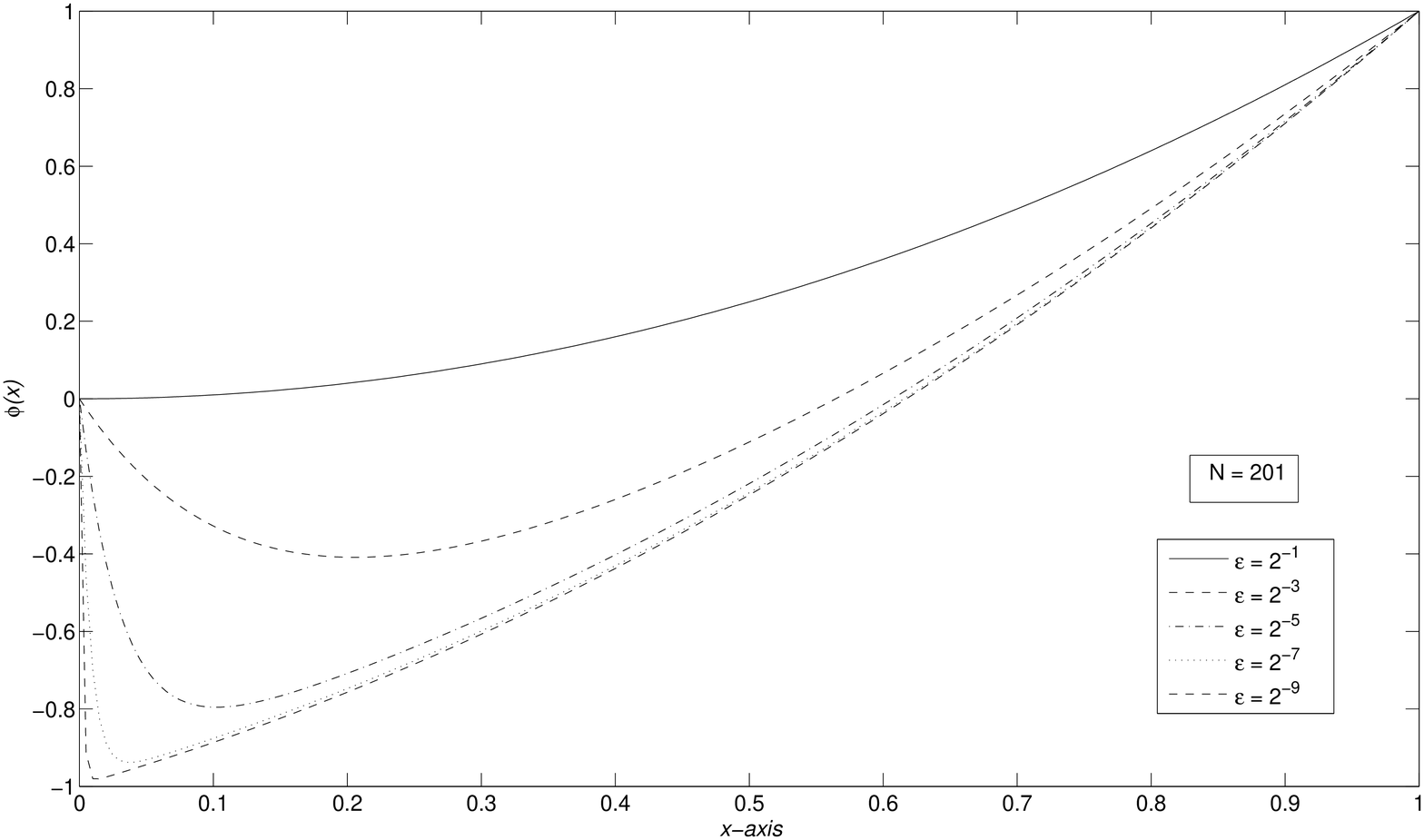}\label{5b}}\\
\subfloat[]{\includegraphics[width=6cm,height=6cm]{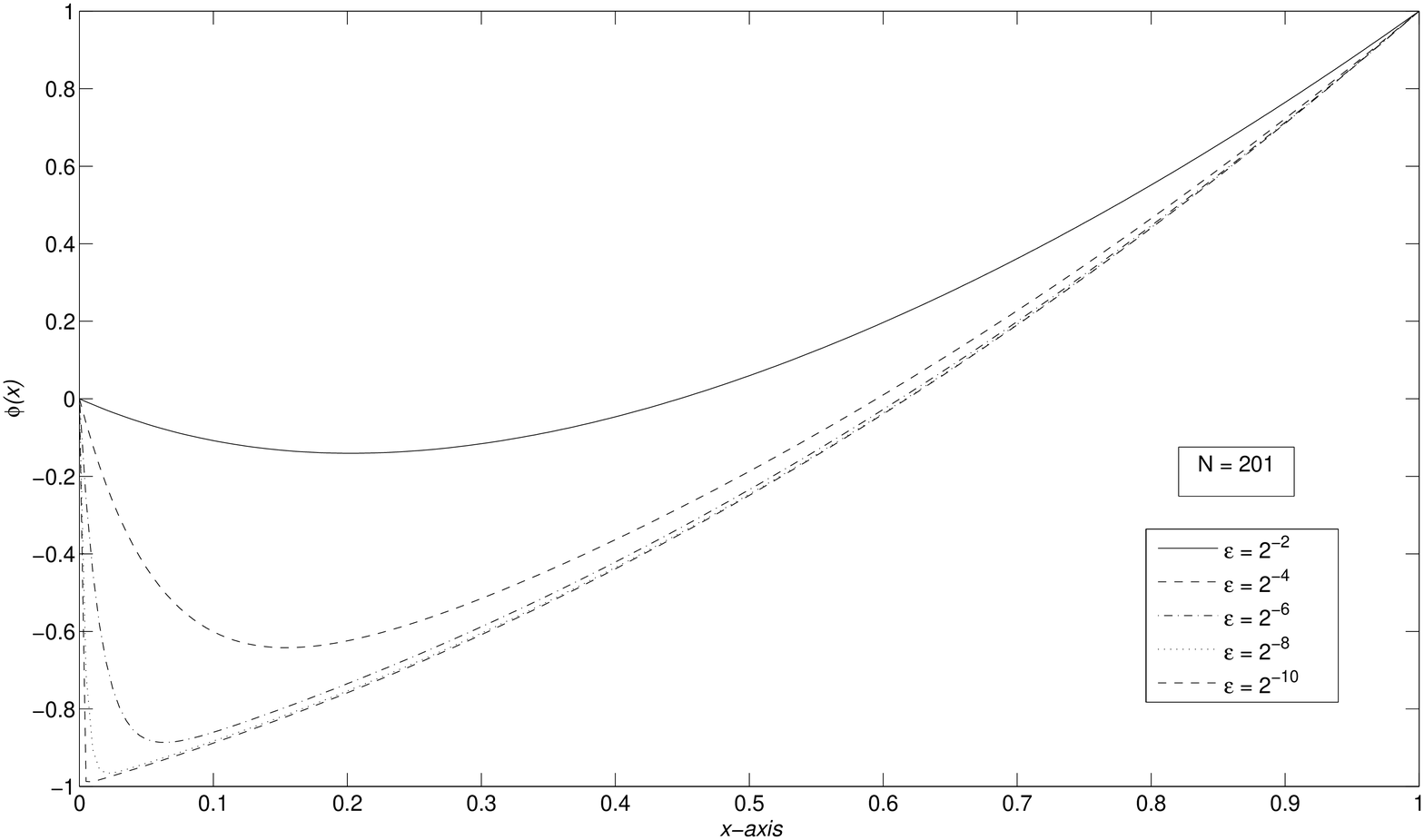}\label{5c}}
\subfloat[]{\includegraphics[width=6cm,height=6cm]{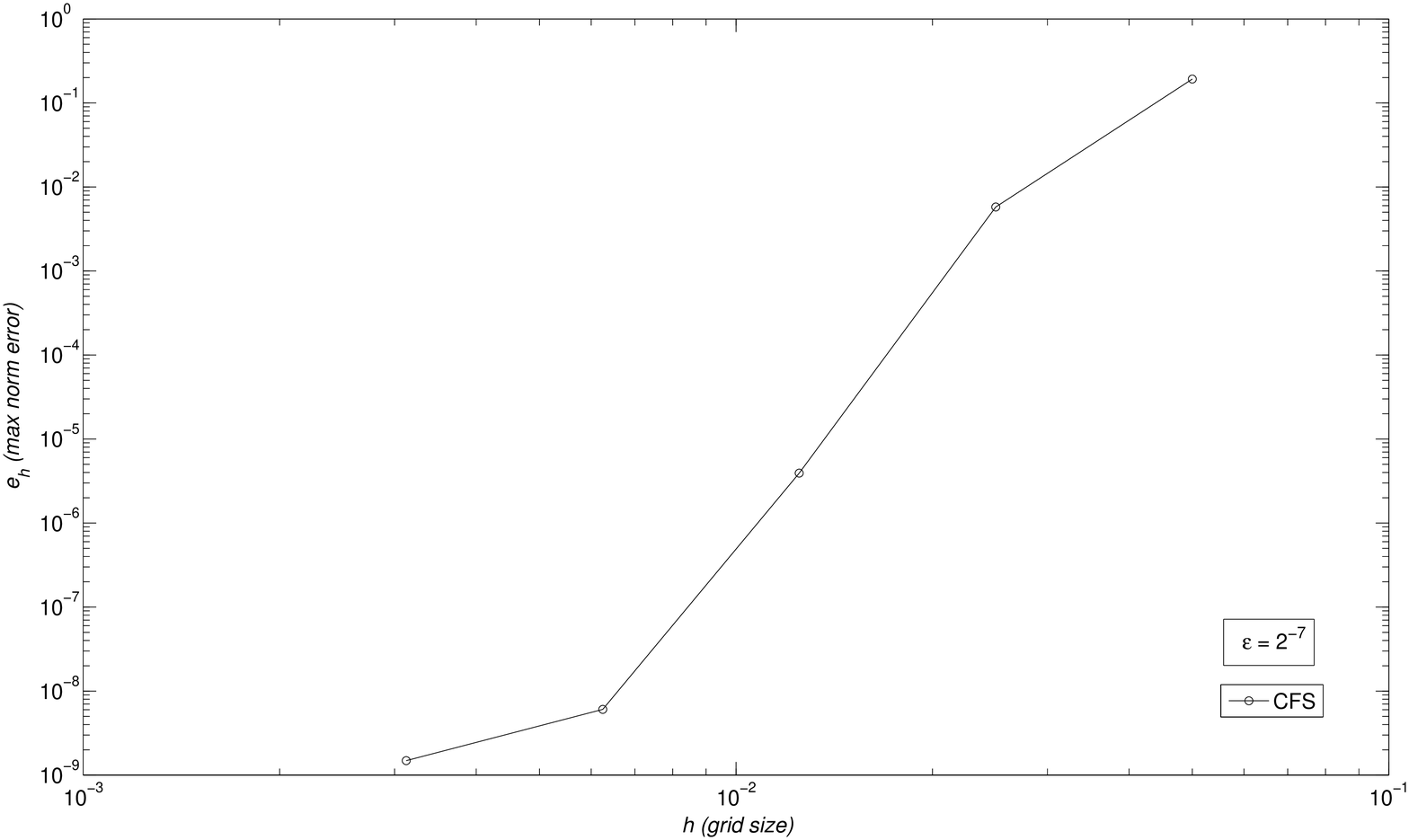}\label{5d}}
\caption{(a) Comparison of the exact $\&$ CFS solutions, (b-c) $\ep$-effects and (d) Log-log graph of the discretization error $e_h$ as a function of the grid size $h$.}
\label{fig:5.5}
\end{figure}

\begin{example} \label{exp:5.6}
Consider the following elliptic SPDDE with appropriate B.C.
\bse \label{eq:5.7.6}
\begin{align}
-\ep &\phi^{\prime\prime}(x) + \phi(x)= s, ~ \fa x \in\Omega, \hspace{2cm} \label{eq:5.7.6a} \\
&\phi(0) = ~ \phi_L,~~\phi(1)= \phi_R,  \label{eq:5.7.6b}
\end{align}
\ese
where $0 < \ep \ll 1$, $b = 0$, $s = x$, $\phi_L = 1$, $\phi_R = 1+ e^{-1/\sqrt\ep}$ and $\Omega = (0,1)$.
The exact solution of the corresponding approximate SPP is given by
\begin{align*}
\phi(x) = e^{-x/\sqrt\ep} + x.
\end{align*}
\end{example}

Here, the solution has a thin boundary layer of width $\ep$ near the boundary $x=1$. The comparison of the exact and numerical solutions is as shown in the figure \ref{fig:5.6}(a). Let $h = \Delta x = 1/(N-1)$ be the grid (step) size with $N$ number of grid points. The number of grid points $N=201$ are adequate to capture the boundary layer accurately, in this example. The effects of different values of singular perturbation parameter $\ep$ are shown in the figures \ref{fig:5.6}(b)-\ref{fig:5.6}(c) and from these figures, one can easily notice that as the singular perturbation parameter $\ep$ goes smaller and smaller, the boundary layers become sharper and sharper. The discretization error $e_h$ is calculated in the max norm for different grid sizes $h = 0.05, 0.025, 0.0125, 0.00625, 0.003125$. Log-log graph for the discretization error $e_h$, showing convergence, is shown in the figure \ref{fig:5.6}(d).

\begin{figure}
\subfloat[]{\includegraphics[width=6cm,height=6cm]{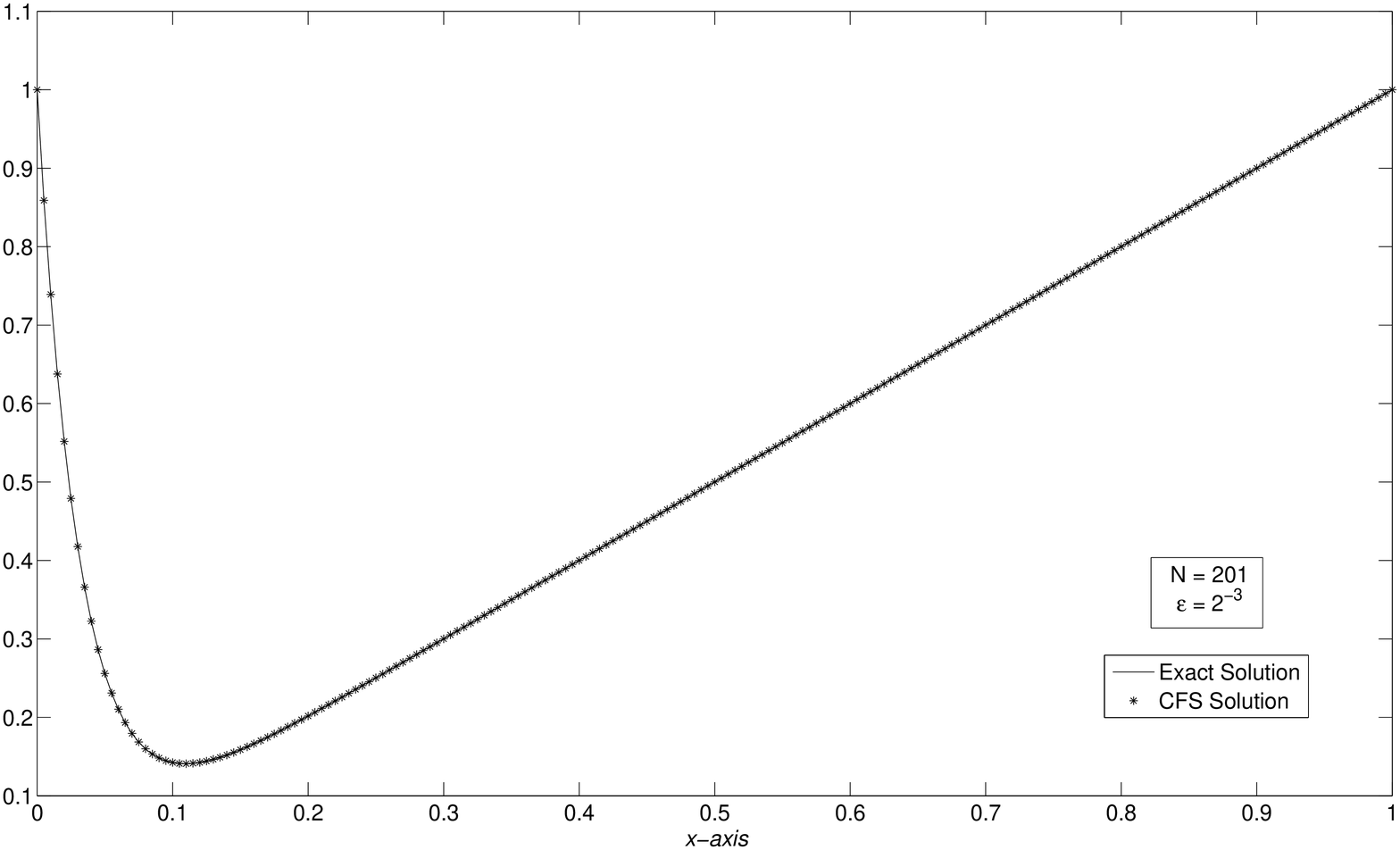}\label{6a}}
\subfloat[]{\includegraphics[width=6cm,height=6cm]{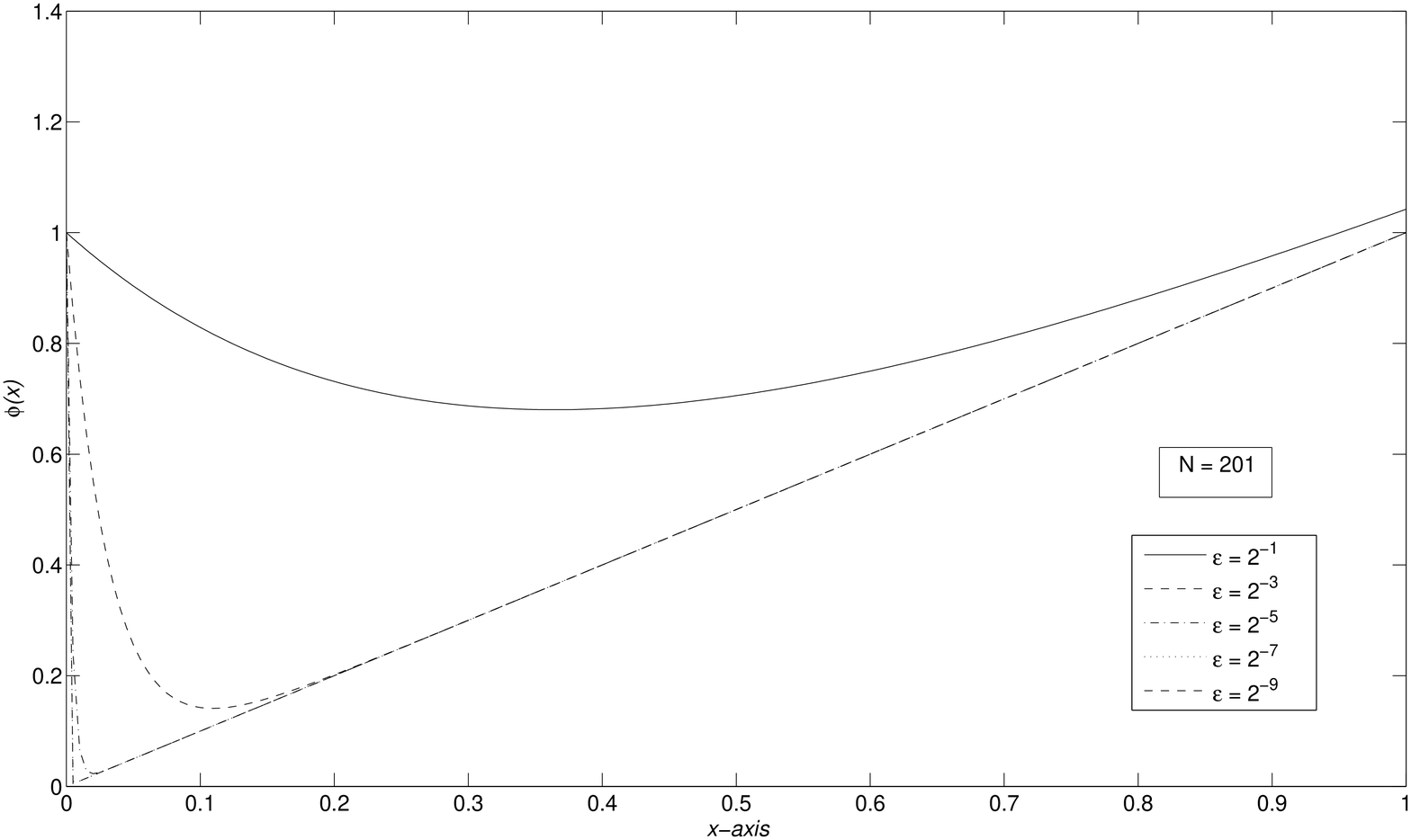}\label{6b}}\\
\subfloat[]{\includegraphics[width=6cm,height=6cm]{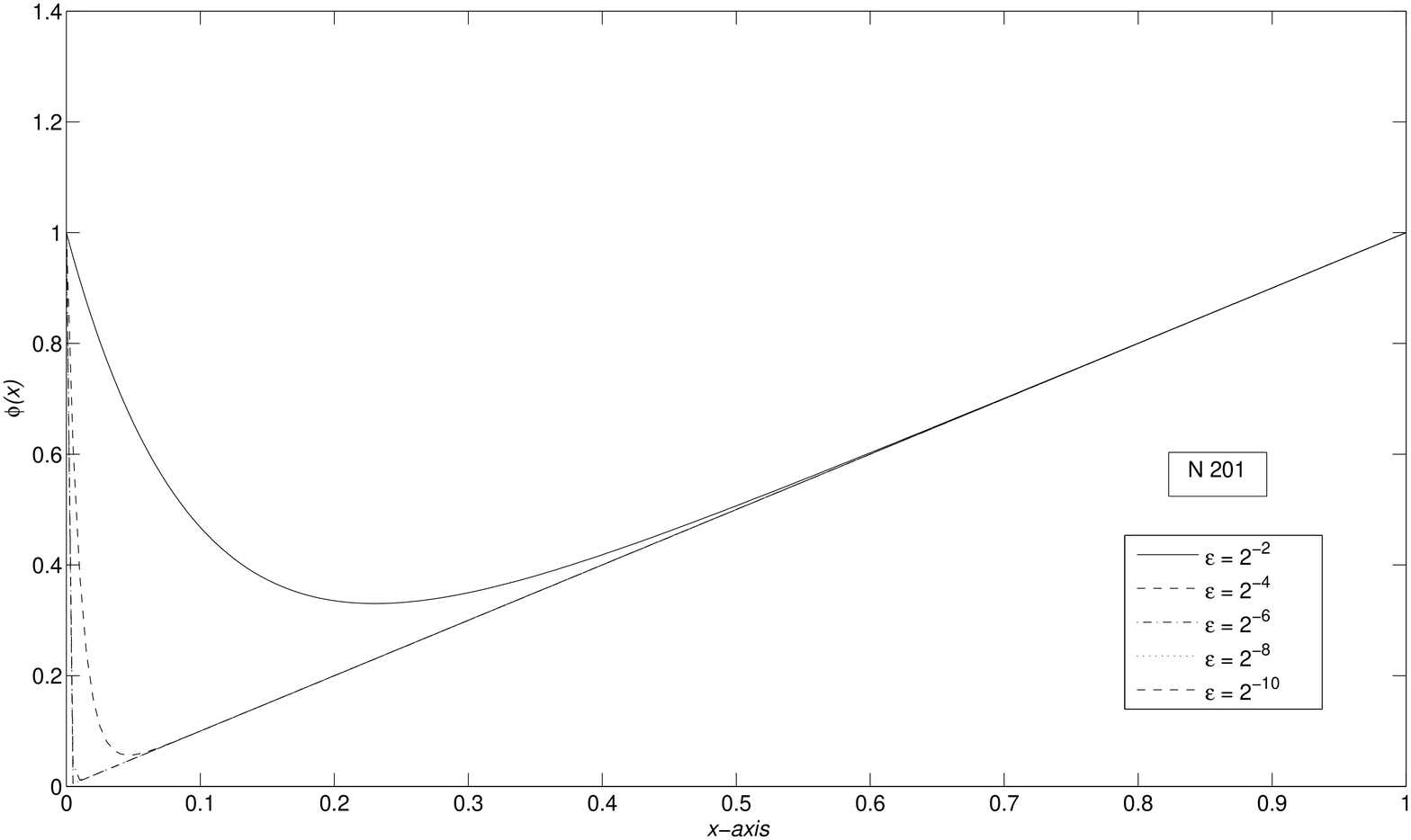}\label{6c}}
\subfloat[]{\includegraphics[width=6cm,height=6cm]{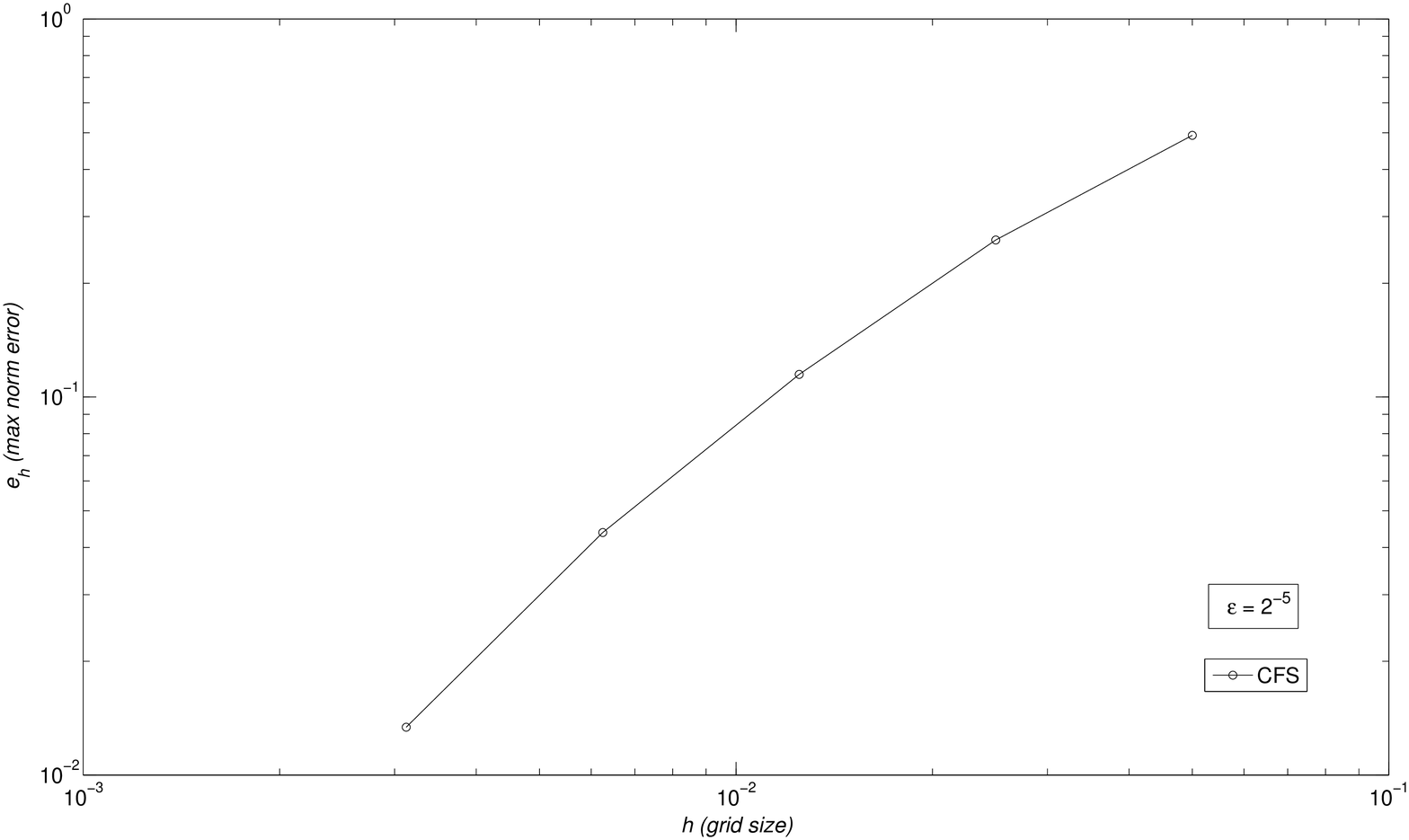}\label{6d}}
\caption{(a) Comparison of the exact $\&$ CFS solutions, (b-c) $\ep$-effects and (d) Log-log graph of the discretization error $e_h$ as a function of the grid size $h$.}
\label{fig:5.6}
\end{figure}

\begin{figure}
\subfloat[]{\includegraphics[width=6cm,height=6cm]{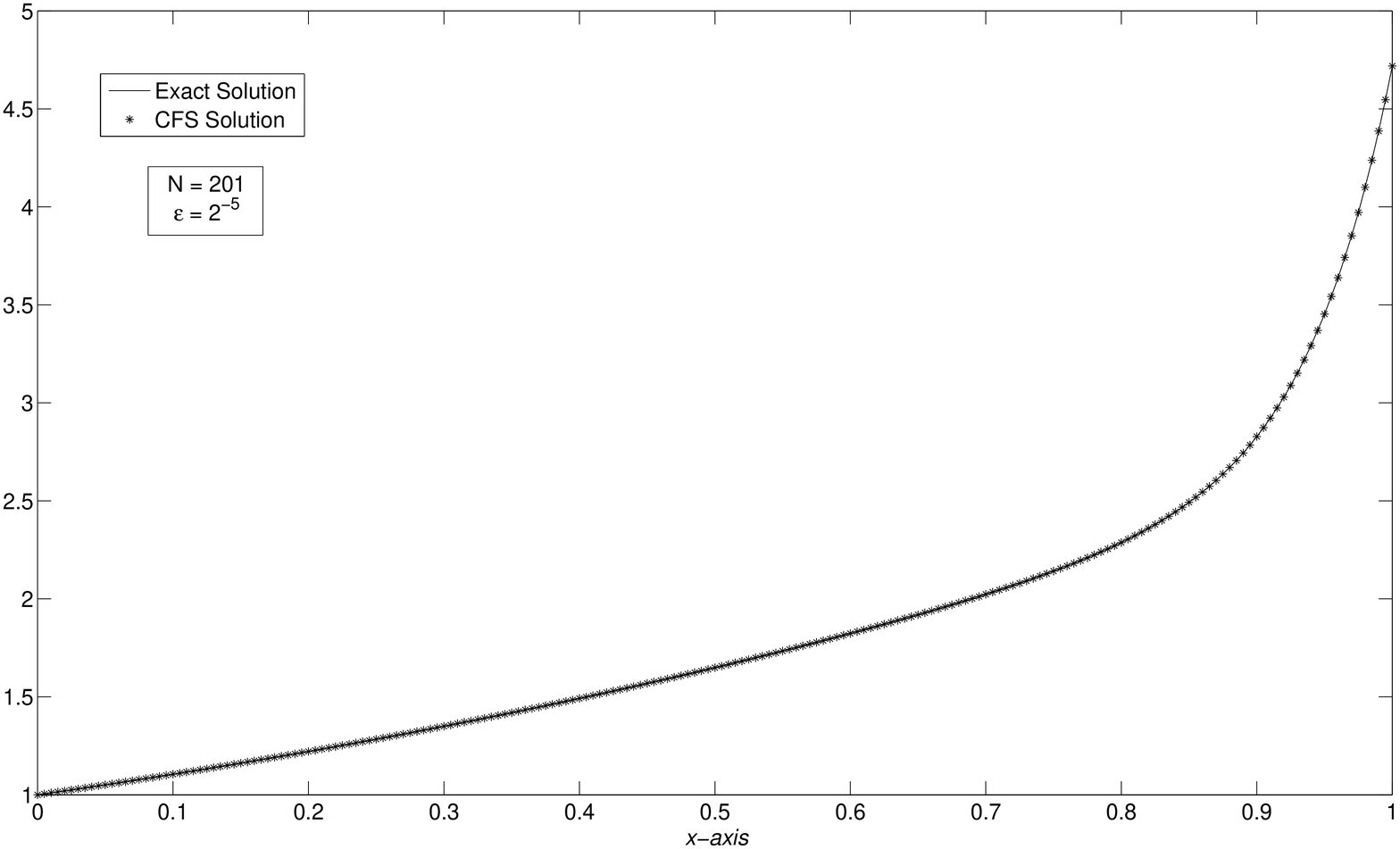}\label{7a}}
\subfloat[]{\includegraphics[width=6cm,height=6cm]{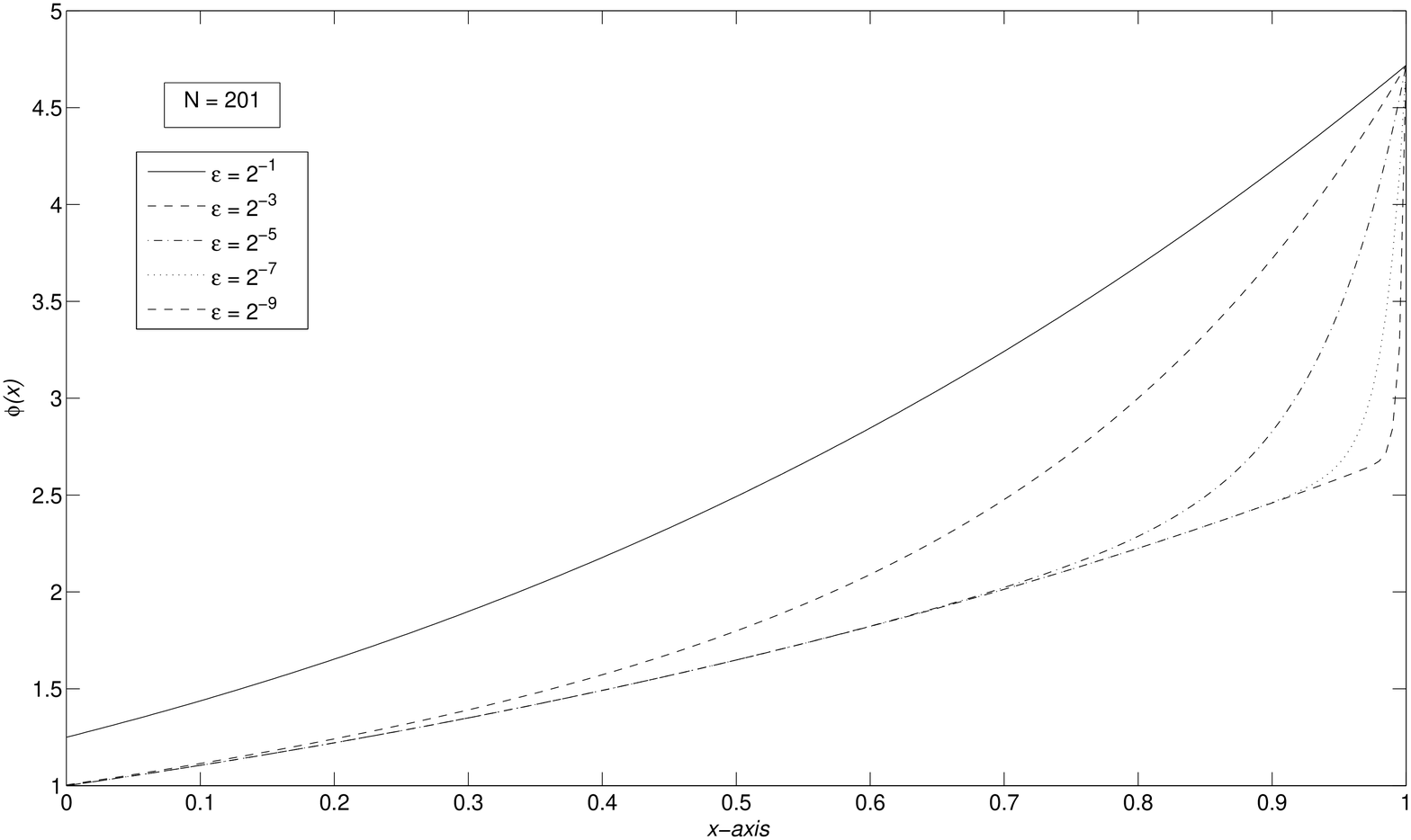}\label{7b}}\\
\subfloat[]{\includegraphics[width=6cm,height=6cm]{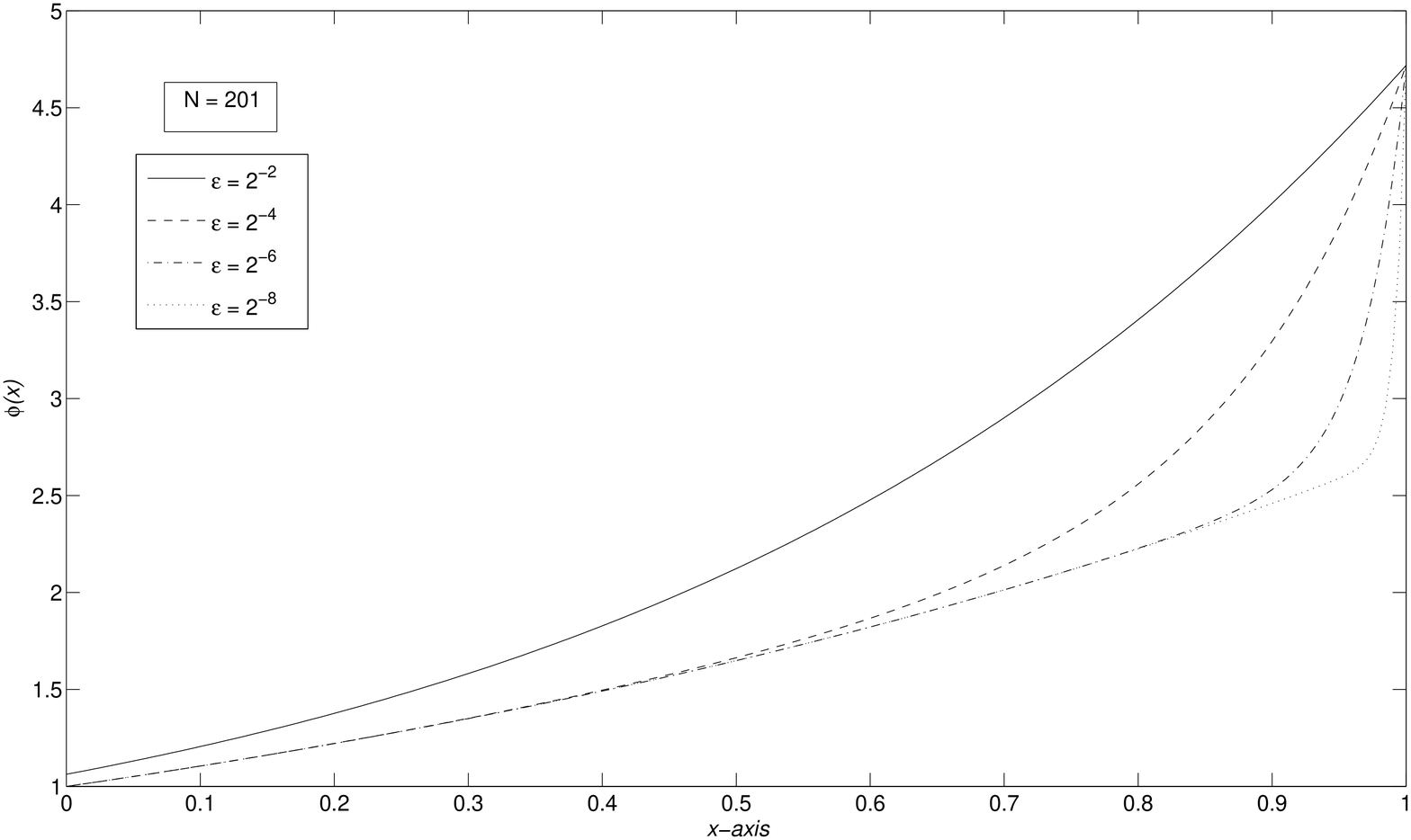}\label{7c}}
\subfloat[]{\includegraphics[width=6cm,height=6cm]{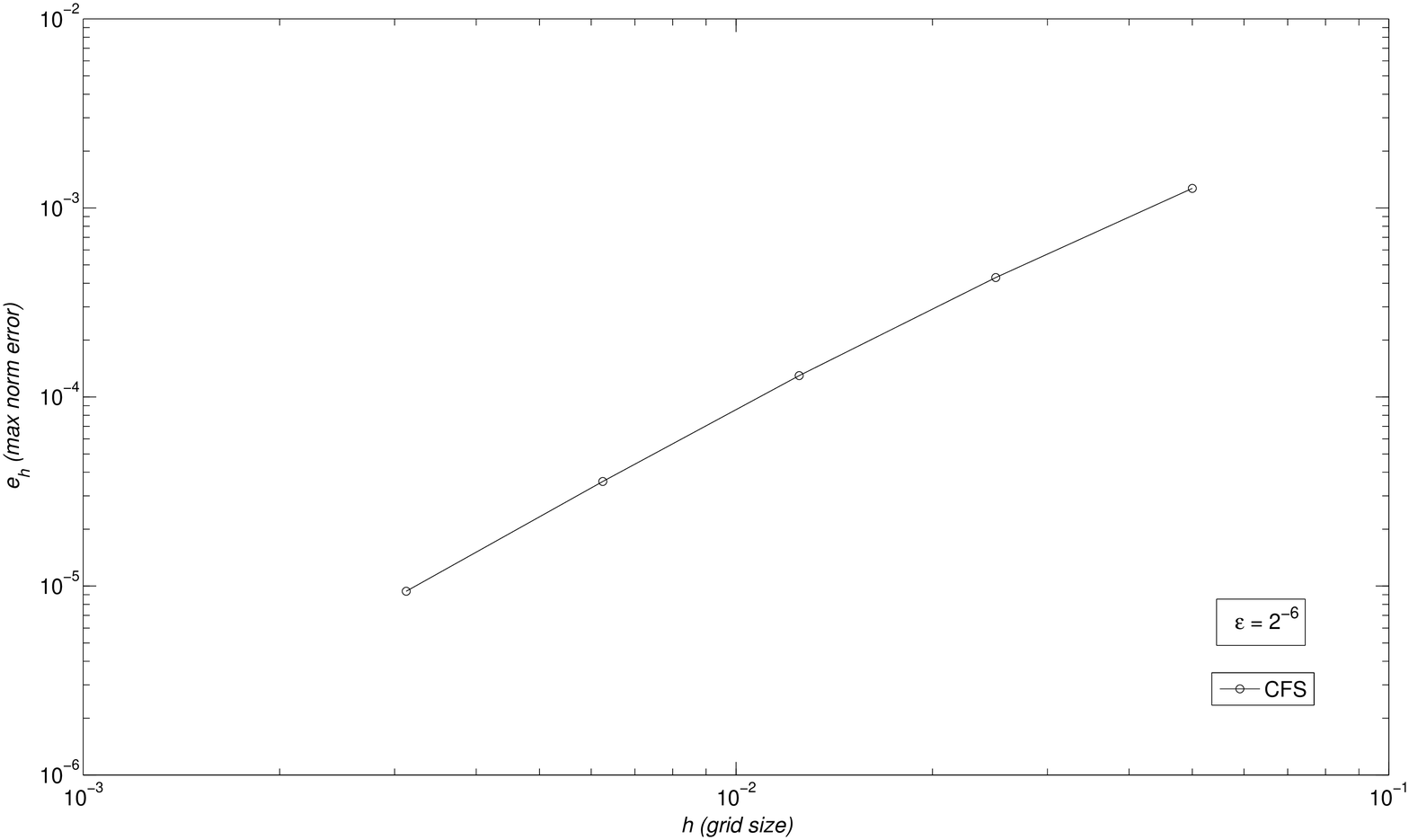}\label{7d}}
\caption{(a) Comparison of the exact $\&$ CFS solutions, (b-c) $\ep$-effects and (d) Log-log graph of the discretization error $e_h$ as a function of the grid size $h$.}
\label{fig:5.8}
\end{figure}

\begin{example} \label{exp:5.8}
Consider the following elliptic SPDDE with appropriate B.C.
\bse \label{eq:5.7.8}
\begin{align}
-\ep &\phi^{\prime\prime}(x) + b\phi^{\prime}(x-\mu) + c \phi(x)= s, ~ \fa x \in\Omega, \hspace{2cm} \label{eq:5.7.8a} \\
&\phi(0) = ~ \phi_L,~~\phi(1)= \phi_R,  \label{eq:5.7.8b}
\end{align}
\ese
where $0 < \ep \ll 1$, $\mu =0$, $b = \frac{1}{x+1}$, $c = \frac{1}{x+2}$, $\phi_L = 1 + 2^{(-1/\ep)}$, $\phi_R = \exp(1) + 2$, $\Omega = (0,1)$ and source term $s$ is so chosen to satisfy the exact solution given by
\begin{align*}
\phi(x) = \exp(x) + 2^{(-1/\ep)} (x+1)^{(1+1/\ep)}.
\end{align*}
\end{example}

Here, the solution has a thin boundary layer of width $\ep$ near the boundary $x=1$. The comparison of the exact and numerical solutions is as shown in the figure \ref{fig:5.8}(a). Let $h = \Delta x = 1/(N-1)$ be the grid (step) size with $N$ number of grid points. The number of grid points $N=201$ are adequate to capture the boundary layer accurately, in this example. The effects of different values of singular perturbation parameter $\ep$ are shown in the figures \ref{fig:5.8}(b)-\ref{fig:5.8}(c) and from these figures, one can easily notice that as the singular perturbation parameter $\ep$ goes smaller and smaller, the boundary layers become sharper and sharper. The discretization error $e_h$ is calculated in the max norm for different grid sizes $h = 0.05, 0.025, 0.0125, 0.00625, 0.003125$. Log-log graph for the discretization error $e_h$, showing convergence, is shown in the figure \ref{fig:5.8}(d).

\section{Conclusion}
\label{sec:5.8}
A complete flux scheme (CFS) for elliptic singularly perturbed differential-difference equations (SPDDEs) has been proposed. Using the source function based element wise inhomogeneous BVPs, fluxes are obtained with distinct two components arising from the homogeneous and particular solutions of the BVPs. The inhomogeneous fluxes written in term of Green's function together with an apt choice of quadrature rules facilitate the derivation of numerical fluxes. The resulting CFS is shown to be stable, consistent and second order convergent. Moreover, the convergence is $\ep$ and $\mu$ uniform. The CFS thus obtained is easy to implement and has direct extension to multi-dimensions.

%
%
%

\section*{References}


\end{document}